\documentclass[reqno,final]{amsart}
\usepackage[utf8]{inputenc}

\usepackage{amsfonts}
\usepackage{amsmath}
\usepackage{amssymb}

\usepackage{bm}
\usepackage{bbm}

\usepackage{geometry}
\usepackage{graphicx}
\usepackage{amsfonts}
\usepackage{amsmath}
\usepackage{amssymb}
\usepackage{fancyhdr}
\usepackage{indentfirst}
\usepackage{booktabs}
\usepackage{verbatim}
\usepackage{color}
\usepackage{amsthm}
\usepackage{url}
\usepackage{cite}
\usepackage{epstopdf}

\usepackage[makeroom]{cancel}

\usepackage[page,toc,titletoc]{appendix}

\usepackage{chngcntr}
\usepackage{apptools}
\AtAppendix{\counterwithin{theorem}{section}}

\usepackage{titletoc}

\usepackage[colorlinks=true]{hyperref} 
\usepackage{caption}

\usepackage[foot]{amsaddr}

\newtheorem{theorem}{Theorem}[section]
\newtheorem{lemma}[theorem]{Lemma}

\newtheorem{hypothesis}[theorem]{Hypothesis}
\newtheorem{remark}[theorem]{Remark}

\newcommand{\rd}{\mathrm{d}}
\newcommand{\p}{\varphi}
\newcommand{\RR}{\mathbb{R}}
\newcommand{\ve}{\varepsilon}

\newcommand*{\e}{\mathop{}\!\mathrm{e}}

\title{On the kinetic description of the objective molecular dynamics}

\author[R. D. James]{Richard D. James $^{\ddagger}$}
	\address{$^{\ddagger}$Department of Aerospace Engineering and Mechanics, University of Minnesota--Twin Cities, Minneapolis, MN 55455 USA. \href{mailto:james@umn.edu}{james@umn.edu} }

\author[K. Qi]{Kunlun Qi $^{\ddagger}$}
	\address{$^{\ddagger}$School of Mathematics, University of Minnesota--Twin Cities, Minneapolis, MN 55455 USA. \href{mailto:kqi@umn.edu}{kqi@umn.edu} }

\author[L. Wang]{Li Wang $^{\ddagger}$}
	\address{$^{\ddagger}$School of Mathematics, University of Minnesota--Twin Cities, Minneapolis, MN 55455 USA. \href{mailto:liwang@umn.edu}{liwang@umn.edu} }

\begin{document}

\begin{abstract}
    In this paper, we develop a multiscale hierarchy framework for objective molecular dynamics (OMD), a reduced order molecular dynamics with a certain symmetry, that connects it to the statistical kinetic equation, and the macroscopic hydrodynamic model. In the mesoscopic regime, we exploit two interaction scalings that lead, respectively, to either a mean-field type or to a Boltzmann-type equation. It turns out that, under the special symmetry of OMD, the mean-field scaling leads to a substantially simplified Vlasov equation that extinguishes the underlying molecular interaction rule, whereas the Boltzmann scaling yields a meaningful reduced model called the homo-energetic Boltzmann equation. At the macroscopic level, we derive the corresponding Euler and Navier-Stokes systems by conducting a detailed asymptotic analysis. The symmetry again significantly reduces the complexity of the resulting hydrodynamic systems.  
\end{abstract}

\maketitle

\tableofcontents
 
\section{Introduction}

\subsection{Objective molecular dynamics}

    Molecular Dynamics (MD) has been the building block for many physical and biological systems. However, even with the modern computational capacity, it is still onerous to simulate a large-scale molecular system. 
    Motivated by the observation that classical molecular dynamics (MD) with certain symmetric properties \cite{DJ2010_OMD} evolves within a smaller manifold, Objective Molecular Dynamics (OMD) aims to leverage the symmetry and invariance of atomic forces (e.g., \eqref{frame}-\eqref{permutation}). OMD can be seen as a specialized form of MD that significantly reduces the computational cost compared to conventional MD.
    It has been applied to the failure of carbon nanotubes under stretching \cite{DJ2010_OMD}, fluid flows with phase transformation \cite{PSJ_2023}, hypersonic flows \cite{pahlani2023objective3} and dislocation 
    motion in crystals \cite{pahlani2023objective1}. 
    
    To explain the idea, consider a structure consisting of $M$ molecules and each molecule consists of $N$ atoms, denoted as 
	\begin{align*}
	    \mathcal{S} := \{x_{i,k} \in \RR^3 : i = 1, \cdots, M, ~ k = 1, \cdots, N \}\,,
	\end{align*}
    where $x_{i,k}$ is the position of atom $k$ in molecule $i$. Then this structure is said to be an {\it objective molecular structure} if 
	\begin{align*}
	    \{ x_{i,k} + Q_{i,k} (x_{j,l}-x_{1,k}): j = 1, \cdots, M,~ l= 1, \cdots, N \} = \mathcal{S}, \qquad \text{for}~ i = 1, \cdots, M, ~ k = 1, \cdots, N.
	\end{align*}
    Here $Q_{i,k} \in O(3)$, where $O(3)$ is the orthogonal group. 
    Putting into words, this requirement means that atom $k$ in molecule $i$ (denoted as atom $(i,k)$ from here on), after re-orientation, sees exactly the same environment as atom $k$ in molecule $1$. Such a property is shared by many crystal structures and alloys. 
	
    The objective molecular structure, along with some invariances in the interactions between molecules, generates an invariant manifold of molecular dynamics. In particular, suppose that the force on atom $(i,k)$ is given by 
	\begin{align*}
	    f_{i,k} (\cdots, x_{j,1}, x_{j,2}, \cdots, x_{j,N}, x_{j+1,1}, x_{j+1,2}, \cdots, x_{j+1,N}, \cdots)\,,
	\end{align*}
	then it is subject to two fundamental invariances: 
	\begin{itemize}
	    \item[i)] Frame indifference. For $Q \in O(3)$, $c \in \RR^3$, 
	    \begin{equation}\label{frame}
        \begin{split}
         & f_{i,k} (\cdots, Qx_{j,1}+c,  \cdots, Qx_{j,N}+c, Qx_{j+1,1}+c,  \cdots, Qx_{j+1,N}+c, \cdots)
	        \\& \quad = Q f_{i,k} (\cdots, x_{j,1},  \cdots, x_{j,N}, x_{j+1,1}, \cdots, x_{j+1,N}, \cdots)\,.
        \end{split}
	    \end{equation}
	    \item[ii)] Permutation invariance. For all permutations $\Pi$, 
	    \begin{equation}\label{permutation}
        \begin{split}
          & f_{i,k} (\cdots, x_{\Pi(j,1)},  \cdots, x_{\Pi(j,N)}, x_{\Pi(j+1,1)},  \cdots, x_{\Pi(j+1,N)}, \cdots)
	        \\& \quad = f_{\Pi(i,k)} (\cdots, x_{j,1},  \cdots, x_{j,N}, x_{j+1,1}, \cdots, x_{j+1,N}, \cdots)\,.
        \end{split} 
	    \end{equation}
	\end{itemize}
	
	Denoting the isometry group in the affine space $\RR^3$ as
	\[
	(Q\,|\,c): \quad Q \in O(3), ~ c \in \RR^3
	\]
	with the product $(Q_1\,|\,c_1)(Q_2\,|\,c_2) = (Q_1 Q_2\,|\,c_1 + Q_1 c_2)$ and inverse $(Q\,|\,c)^{-1} = (Q^\top|-Q^\top c)$. Then its action on  $\RR^n$ can be written as 
	\begin{align*}
	    g(x) = Q x + c, \quad x \in \RR^3\,,
	\end{align*}
    or $g =: (Q\,|\,c)$ for short. Now if we assume $Q_1$, $\cdots$, $Q_{M}$ to be constant matrices while allow $c_1$, $\cdots$, $c_{M}$ to have the following time dependence 
    \begin{align*}
        c_i(t) = a_i t + b_i, \quad a_i, b_i \in \RR^3\,, \quad i = 1, \cdots, M\,,
    \end{align*}
    then it is obvious that for any $x(t)$, we have
    \begin{align*}
        \frac{\rd^2}{\rd t^2} g_i(x(t))= \frac{d^2}{\rd t^2} (Q_i x(t) + c_i(t))
        = Q_i \frac{\rd^2 x (t)}{\rd t^2}\,,
    \end{align*}
    which, together with the invariance above, implies the existence of a time-dependent invariant manifold of equations of molecular dynamics. 
    
    Building upon such an invariant manifold, the OMD works as follows. It divides $MN$ atoms into $N$ simulated particles---denoted as $(1,1),~ \cdots ,~ (1,N)$, conceptually the atoms in molecule $1$; and $(M-1)N$ non-simulated particles---denoted as $(2,1),~ \cdots ,~ (2,N), ~\cdots, (M,1), ~\cdots, ~(M, N)$, atoms in molecules $2, \cdots M$. Their positions have the relation
    \begin{align} \label{OMD-2}
        x_{i,k}(t) = g_i \big(x_{1,k}(t),t\big)\,, ~ g_i = \big(Q_i \,|\, c_i(t)\big),\qquad  i = 1, \cdots, M, \ k = 1, \cdots, N\,,
    \end{align}
    and $g_1 = \text{id}$. Then the simulated particles move according to the following rule:
    \begin{align} \label{OMD-1}
        m_{k} \frac{\rd^2}{\rd t^2} x_{1,k} = f_{1,k} (\cdots, x_{j,1}, x_{j,2}, \cdots, x_{j,N}, x_{j+1,1}, x_{j+1,2}, \cdots, x_{j+1,N}, \cdots)\,,
    \end{align}
    whereas the non-simulated particles update directly via \eqref{OMD-2}. The basic theorem of OMD states that each non-simulated atom satisfies exactly the equations of molecular dynamics for its forces. This way, the total number of degrees of freedom is substantially reduced and therefore leads to a much more efficient computational method. That is, if a cut-off for the atomic forces is introduced, only the simulated atoms, together with the non-simulated atoms within the cut-off, need to be tracked.
    Despite the positions of the non-simulated atoms being given by explicit formulas, the overall motion is typically highly chaotic.
	

With the number of particles getting large, a coarse-grained model, termed as a kinetic equation, is introduced to give a statistical description of the collective behavior of the many-particle system.

	 
\subsection{Motivation and previous results}

There is now a full-fledged theory on the derivation of the kinetic and hydrodynamic equations. 
 This theory focuses on particle interactions within a classically {\it unstructured} background. In this paper, our primary objective is to establish a multi-scale framework that emphasizes symmetry. Since OMD represents an invariant manifold of MD, it is important to know whether this manifold is in some sense inherited in reduced-order kinetic equations.  We aim to establish such a systematic connection to reduced-order kinetic equations and their corresponding macroscopic models.

\begin{figure}[ht]
	    \centering
	    \includegraphics[width=0.8\textwidth]{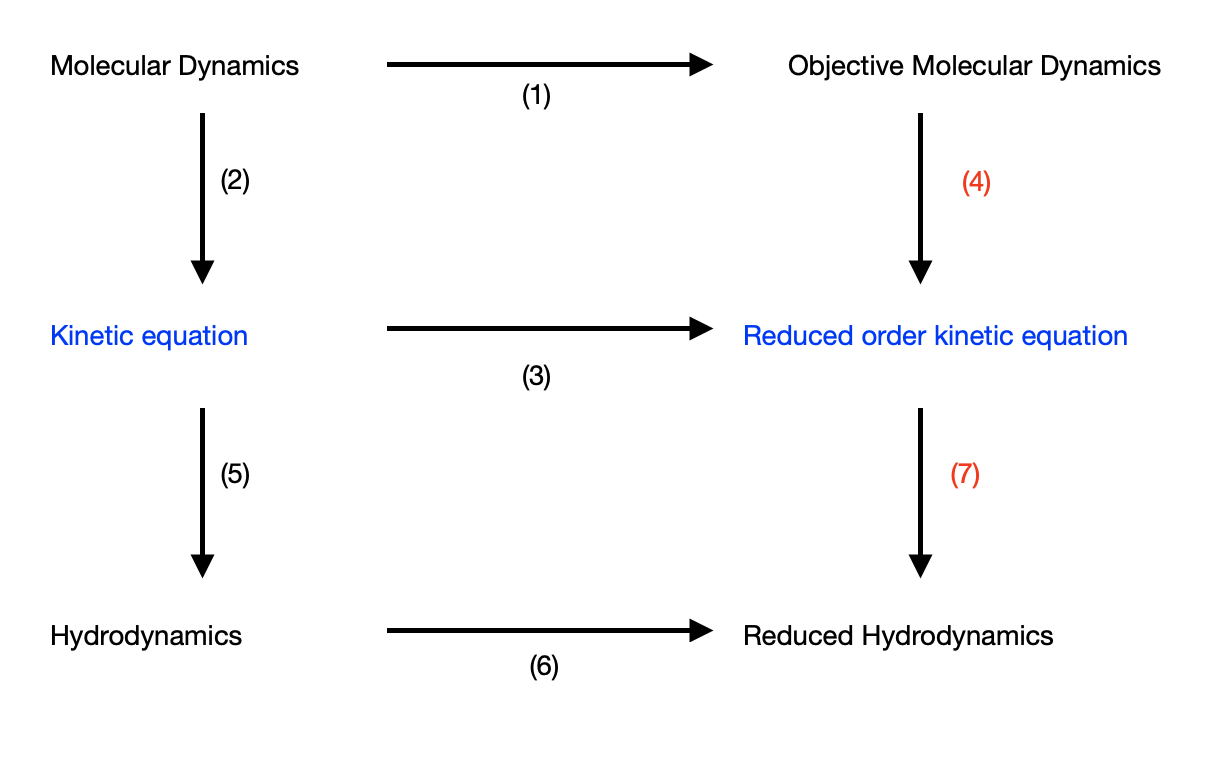}
	    \caption{Connections between different models in the multi-scale hierarchy.}
	    \label{fig:diagram}
\end{figure}

 To provide a more compelling representation of our motivation and results, we use Fig.~\ref{fig:diagram} for an illustration.
 
\begin{itemize}
    \item \textbf{Arrow (1)} has been explained above, see also \cite{JNV2019, James_ICM}.\\
    
    \item \textbf{Arrow (2)} is a well-established relationship in kinetic theory. Instead of tracking the detailed motions of each molecule in the dynamic system, which is computationally impractical due to the enormous number of particles, the kinetic equation allows us to analyze the system's behavior without considering individual particle motions.

   To achieve this, the BBGKY hierarchy (from the names of Bogolyubov, Born, Green, Kirkwood, Yvon) has been proven to be a useful methodology \cite{CIP}. Additionally, suitable scaling limits are employed to capture the essential properties of the microscopic regime. Two typical scalings, the \textit{mean-field limit} and the \textit{Boltzmann-Grad limit}, have led to the two different types of kinetic equations. 
   
   The mean-field limit, stemming from \cite{Braun_Hepp_1977}, assumes that the force on one particle is influenced by the entire range of other particles, although the strength of interaction weakens as the number of particles $N$ increases. As $N$ approaches infinity, a Mean-field/Vlasov-type equation emerges, where the particle distribution depends on its phase space density. For a comprehensive review on this topic, refer to \cite{Spohn2012large, Golse2003}.

    On the other hand, the Boltzmann-Grad limit arises when the particles are diluted enough that only binary interactions play a significant role, and each particle experiences a single collision within a given unit of time \cite{Villani02}.
    In this case, the Boltzmann equation is formally derived by Grad and Cercignani \cite{Grad1949, Grad1958, Cercignani1972}, with rigorous validation by Lanford \cite{Lanford1975} for the hard-sphere model over short times. Extensive studies have been conducted on smooth short-range potentials \cite{King1975, GST2013, PSS2014}, and for a recent review, see \cite{PS2021}.\\

    \item \textbf{Arrow (3)} has been discussed in \cite{DJ2010_OMD, James_ICM}, either via heuristic argument of symmetry in statistical physics language, or by looking for a special ansatz of solution that reduces the equation. 

    To be more specific, recall the classical Boltzmann equation:
	\begin{equation}\label{BE}
		\partial_{t} f(t,x,v) +v\cdot \nabla_x f(t,x,v)=\mathcal{Q}(f,f)(t,x,v), \quad t>0,  \ x\in\mathbb{R}^3, \ v\in\mathbb{R}^3,
	\end{equation}
    where
	\begin{equation}\label{Q}
	\mathcal{Q}(f,f)(t,x,v)=\int_{\mathbb{R}^3}\int_{\mathbb{S}^{2}} B(v-v_*,\sigma)\left[f(t,x,v')f(t,x,v_*')-f(t,x,v)f(t,x,v_*)\right] \,\rd\sigma \,\rd v_*\,,
	\end{equation}
    with the collision kernel $B(v-v_*,\sigma)$ that describes the intensity of collisions. Usually, $B$ can be separated as the kinetic part $\Phi$ and angular part $b$ in the case of the inverse power law:
    \begin{equation*}
		B(|v-v_*|,\sigma) = b(\cos\theta) \Phi(|v-v_*|), \quad \text{with} \ \cos\theta=\sigma \cdot \frac{v-v_*}{|v-v_*|},
    \end{equation*} 
    where kinetic collision part $\Phi(|v-v_*|)=|v-v_*|^{\gamma}$ includes hard potential $ (\gamma>0) $, Maxwellian molecule $ (\gamma =0) $ and soft potential $ (\gamma<0) $, and angular part $b(\cos\theta)$ is often regarded to satisfy the Grad's cutoff assumption, i.e.,
    $\int_{\mathbb{S}^{2}} b(\cos\theta) \,\rd \sigma < \infty$,
    see more details of the collision kernel $B$ in \cite{Villani02}. Here $(v', v'_*)$ and $(v, v_*)$ represent the velocity pairs before and after the collision, respectively. They satisfy the conservation of momentum and energy:
	\begin{equation*} 
		v' + v_{*}' = v + v_{*}, \quad  |v'|^{2} + |v_{*}'|^{2} = |v|^{2} + |v_{*}|^{2}.
	\end{equation*}
    This allows us to express $(v', v_*')$ in terms of $(v, v_*)$ using the following equations:
    \begin{equation*}
		v'=\frac{v+v_*}{2}+\frac{|v-v_*|}{2}\sigma, \quad  v_*'=\frac{v+v_*}{2}-\frac{|v-v_*|}{2}\sigma\,.
    \end{equation*}

    Now translating the OMD symmetry in kinetic language, it means that \cite[p.~155]{DJ2010_OMD} ``the probability of finding a velocity of the form $v + A(I+tA)^{-1}x$ at $x$ is the same as the probability of finding a velocity $v$ at $0$". Putting the words into a formula, we have 
    \begin{equation*}
        f(t,x,v+A(I+tA)^{-1}x) = f(t,0,v)\,,
    \end{equation*}
    which is equivalent to 
    \begin{equation}\label{vw}
        f(t,x,v) = f(t,0, v-A(I+tA)^{-1}x) =: g(t,w) \quad \text{with} \quad w = v - A(I+At)^{-1}x\,.
    \end{equation}
    Here  $A$ is an assignable 3x3 matrix.
    Then $g$, which depends on fewer variables, satisfies the homo-energetic Boltzmann equation, a reduced order model originally introduced by Galkin \cite{Galkin1958} and Truesdell \cite{Truesdell1956}: 
        \begin{equation}\label{g}
	\begin{split}
	\partial_{t} g(t,w) -  [A(I+tA)^{-1} w] \cdot \nabla_{w} g(t,w) = \mathcal{Q}(g,g)(t,w).
	\end{split}
    \end{equation}

   An alternative approach involves seeking the equi-dispersive solution to Eq.~\eqref{BE} \cite{JNV2019}. In other words, we look for the solution ansatz: if $f$ is the solution to Eq.~\eqref{BE} and
    \begin{equation}\label{ansatz}
        f(t,x,v)= g(t,w) \quad \text{with} \quad  w = v- \xi(t,x)\,,
    \end{equation}
    then $g$ satisfies 
    \begin{equation}\label{gg1}
        \partial_{t} g(t,w) - [\partial_t \xi + \xi\cdot \nabla_x \xi]\nabla_{w} g(t,w) - [(\nabla_x \xi) w] \cdot \nabla_{w} g(t,w) = \mathcal{Q}(g,g)(t,w).
    \end{equation}
    Clearly, by a direct calculation, when $\xi(t,x)$ is an affine function on $x$ such that
    \begin{equation}\label{xi}
		\xi(t,x) = L(t)x, \quad \text{with} \quad L(t):= A(I+At)^{-1}x,
	\end{equation}
    the Eq.~\eqref{gg1} can be reduced to Eq.~\eqref{g}.\\

    \item \textbf{Arrow (5)} is the process that leads from kinetic equations in the mesoscopic regime to continuum equations in the macroscopic regime. This concept can be traced back to Maxwell and Boltzmann, who initially founded the kinetic theory. The study of the hydrodynamic limit was subsequently formulated and addressed by Hilbert \cite{Hilbert1912}. It aims to derive the fluid dynamic system as particles undergo an increasing number of collisions, causing the Knudsen number to approach zero.
    
    The classical compressible Euler and Navier-Stokes equations can be formally derived from the scaled Boltzmann equation through the Hilbert \cite{Hilbert1912} and Chapman-Enskog expansions \cite{CC, enskog1917}. The asymptotic convergence of these derivations was rigorously justified by Caflish \cite{Caflisch1980} for the compressible Euler equations and by De Masi, Esposito, and Lebowitz \cite{MEL1989} for the incompressible Navier-Stokes equations.

    Another aspect of studying the hydrodynamic limit pertains to weak solutions, particularly proving that the renormalized solution of the Boltzmann equation converges to the weak solution of the Euler or Navier-Stokes equations. This has been partially achieved for incompressible models \cite{BGL1993, LM2001, Saint-Raymond2003, GS2004, GS2009, LM2010, JM2017}.

    Additionally, research on strong solutions near equilibrium is another avenue of exploration in the hydrodynamic limit. Nishida \cite{Nishida1978} established local-in-time convergence to the compressible Euler equations, while Bardos and Ukai \cite{BU1991}, as well as more recent work by Gallagher and Tristani \cite{GT2020}, derived solutions for the incompressible Navier-Stokes equations. For a comprehensive review of this topic, we refer to \cite{Cercignani00, Golse2003, Saint-Raymond_09} and the references cited therein.

    The combined transitions (2)+(5) form the central framework of Hilbert's Sixth Problem, aiming to establish a comprehensive depiction of gas dynamics across all levels of description. The objective is to comprehend macroscopic concepts such as viscosity and nonlinearity from a microscopic standpoint \cite{Saint-Raymond_09}.\\

    \item \textbf{Arrow (6)} is a heuristic derivation in the case of the macroscopic motion corresponding to an OMD simulation \cite{DJ2010_OMD, PSJ_2022}. The macroscopic velocity field of such a motion is
    $u(x,t) = A(I+ tA)^{-1}x$ and, by direct
    substitution, this is an exact solution of the compressible or incompressible Navier-Stokes equations. For the latter
    add the restriction of incompressibility, $\text{Tr}[(A(I + tA)^{-1}] =0$.  If thermodynamics is included with the former, the energy equation becomes an ODE for the temperature. 
    
\end{itemize}

As mentioned earlier, our primary focus is on the completion of the diagram by establishing connections (4) and (7). The challenges include:
\begin{itemize}
    \item[1)]  In the context of OMD, there exist two sets of particles: simulated and unsimulated. The simulated particles are updated based on Newton's second law, while the unsimulated ones undergo updates through a "copy and paste" mechanism. This fundamental distinction is the primary factor contributing to the significant speedup achieved by OMD. However, when deriving the corresponding kinetic system (i.e. route (4)), a crucial question arises: should these two sets of particles be treated differently? Using the previous notation $x_{i,k}$ with $i = 1$ and $k = 1, \cdots, N$ representing the simulated particles, and $i = 2, \cdots, M$ and $k = 1, \cdots, N$ representing the non-simulated particles, the question arises: should we allow $M$ to approach $\infty$ first, or $N$, or both simultaneously? Furthermore, is there a particular relationship between $N$ and $M$ that is crucial to achieving a meaningful limit?

   \item[2)] In theory, two scalings can be applied in route (4): mean-field scaling or Boltzmann-Grad scaling. 
   Conceptually, both mean-field scaling and Boltzmann-Grad scaling make sense in deriving the mesoscopic (a.k.a. kinetic) models from microscopic particle dynamics,  but they emphasize different kinds of interactions at the particle level:  
\begin{itemize}
    \item The mean-field limit highlights the long-range interactions between particles by assigning each particle an equal weight of influence, denoted as $\frac{1}{N}$, on any given particle.  As a result, a non-local (in $x$) model is anticipated in the mean-field limit; 
    \item The Boltzmann-Grad limit, on the other hand, emphasizes the local interaction. The rescaling of the interaction from $\nabla U(r)$ to $ \frac{1}{\varepsilon} \nabla U\left(\frac{r}{\varepsilon}\right)$ indeed implies that each particle is unaffected by others unless those particles fall within its influence range, the $\varepsilon$-neighborhood.
    \end{itemize} 
   However, in our work, due to the intrinsic symmetry and invariance of the potential in OMD, the expected nonlocal term accounting for interactions among particles in the mean-field limit vanishes. This results in an oversimplified model, as seen in Eq.~\eqref{PR}. On the contrary,  the interactions in the Boltzmann-Grad scaling can be retained, which then leads to the more physically reasonable homo-energetic Boltzmann equation. This circumstance is unexpected until one delves into the derivation process.

   \item[3)] Route (7) that we aim to establish is a lot more formal compared to route (5)+(6), which is accomplished only in a heuristic manner \cite{DJ2010_OMD, PSJ_2022}. This undertaking is challenging compared to the classical Boltzmann equation, due to the intricate handling required for the hyperbolic term on the left-hand side of the homo-energetic equation.
\end{itemize}

We want to emphasize that, despite obtaining seemingly similar results as one could from another route, namely (4) = (2) + (3) and (7) = (5) +(6), this equivalence is far from obvious. Assessing whether symmetry at the OMD level can be maintained at the kinetic and hydrodynamic levels is a non-trivial task. Unlike previous approaches where (2) + (3) or (5) + (7) is {\it assumed} to preserve such symmetry, our approach adopts a more formal derivation without making such assumptions. We systematically investigate how this symmetry is retained throughout the derivation.


	
\subsection{Basic set-up and our results}

Prior to discussing the kinetic formulation, we first lay out some preliminaries. As mentioned before, we denote $(x_{1,k}(t) , v_{1,k}(t))$ as the simulated particles' location and velocity, and $(x_{i,k}(t) , v_{i,k}(t) )$ for the associated non-simulated particles, $i = 2, \cdots, M$. In general terms, their relation is formalized by Eq.~\eqref{OMD-2}. Throughout this paper, we consider the simplest case of the time-dependent translation group \cite{James_ICM} in which case $Q_i = I$ and $c_i(t) = \sum_{l=1}^3 \nu_i^l(b_l + a_l t) = (I+tA) \nu_i$, where $\nu_i = (\nu_i^1, \nu_i^2, \nu_i^3) \in \mathbb Z^3,\, b_l=e_l, \,  a_l = A e_l$ for orthonormal basis $e_l$ and $i = 1, \cdots, M$, then the relation \eqref{OMD-2} reduces to 
\begin{align}\label{OMD-s}
	x_{i,k}(t) = x_{1,k}(t) + (I + tA) \nu_i\,.
\end{align}
Therefore,
 \begin{align*}
     v_{i,k}(t) = v_{1,k}(t) + A \nu_i\,,
 \end{align*}
 which immediately leads to 
 \begin{align*}
     v_{i,k}(t) - v_{1,k}(t) = A (I + t A)^{-1}(x_{i,k}(t) - x_{1,k}(t)) \,.
 \end{align*}
Consequently, define the transformation  
	\begin{align*}
	 w(t) := v(t) - A(I+tA)^{-1}x(t) \,,
	\end{align*}
     then a very important observation is that the simulated and non-simulated particles will be {\it indistinguishable} if written in terms of new variables $w_{1,k}$ and $w_{i,k}$:
	\begin{equation}\label{ww2}
		w_{i,k}(t) = v_{i,k}(t) - A(I+tA)^{-1}x_{i,k}(t) = v_{1,k}(t) - A(I+tA)^{-1}x_{1,k}(t) = w_{1,k}(t)\,,
	\end{equation}
 for $i = 1, \cdots, M$ and $k = 1, \cdots, N$.
 Therefore, for the brevity of notation, we will use the one single customary subscript $i = 1, \cdots, N$ to index all particles throughout the rest of the paper.
\footnote{Note that when $Q_i \neq I$, \eqref{OMD-s} becomes
\begin{align*}
	x_{i,k}(t) = Q_i x_{1,k}(t) + (I + tA) \nu_i\,,
\end{align*}
which then leads to the relation
\begin{align*}
  v_{i,k}(t) - A(I+tA)^{-1}x_{i,k}(t) = Q_iv_{1,k}(t) - A(I+tA)^{-1} Q_ix_{1,k}(t)  \,. 
\end{align*}
To proceed, it is essential to identify a suitable change of variable that unifies the dynamics of simulated and non-simulated particles. In this context, further investigation into the properties of $Q$ and $A$ and their relationship is required. We defer this exploration to future work.} Note that in the following derivation of kinetic limit, the number of atoms $N$ essentially goes to infinity, which leads the total number of particles $MN$ in the system to infinity as well.

    Since
	\begin{equation*}
	    w_{i}(t) : =v_{i}(t) - A(I+tA)^{-1}x_{i}(t)
	\end{equation*}
	and 
    \begin{equation*}
        \dot{x}_{i}(t) = v_i(t) = w_{i}(t) + A(I+tA)^{-1}x_{i}(t)\,,
    \end{equation*}
    a direct calculation shows that 
	\begin{equation*}
		\begin{split}
			\frac{\rd }{\rd t} [A(I+tA)^{-1}x_{i}(t)] =& A(I+tA)^{-1}[\dot{x}_{i}(t) - A(I+tA)^{-1}x_{i}(t)]\\
			=& A(I+tA)^{-1}[v_{i}(t) - A(I+tA)^{-1}x_{i}(t)]\\
			=& A(I+tA)^{-1} w_{i}(t)\,,
		\end{split}
	\end{equation*}
	and therefore 
	\begin{equation}\label{dw2}
	\begin{split}
	    \dot{w}_{i}(t) =& \dot{v}_{i}(t) - A(I+tA)^{-1} w_{i}(t)\\
	    =& - \sum_{\substack{j=1\\j \neq i}}^{N} \nabla_{x_i}U(|x_{i}(t) - x_{j}(t)|) - A(I+tA)^{-1} w_{i}(t)\,,
	\end{split}
	\end{equation}
	where we have used a specific form of the force:
    \begin{align} \label{U}
	f_{1,i} (\cdots, x_{j,1}, x_{j,2}, \cdots, x_{j+1,N}, \cdots) = - \sum_{j = 1}^N  \nabla U_{x_i}(|x_{j} - x_{i}|)\,.
	\end{align}

\begin{remark}
    It is important to note that while the radial condition of the potential function $f_{1,i}$ may not be necessary for OMD at the microscopic level (some symmetric condition such as permutation invariance Eq.~\eqref{permutation} is indispensable), we will focus exclusively on radial potential Eq.~\eqref{U} when deriving the kinetic equation for the remainder of this paper. This choice is mainly due to technical reasons. 
    
    More specifically, in Section \ref{sub:rigorous}, where we derive the mean-field limit, the radial potential is necessary for proving the well-posedness of the mean-field equation (Eq.~\eqref{P1UR}), following \cite{CCR2011}. Additionally, our derivation of the homo-energetic Boltzmann equation in Section \ref{sec:boltzmann} closely follows the classical approach using the Boltzmann-Grad limit described in \cite{GST2013}, where the radial condition is also required \cite[Assumption 1.2.1]{GST2013}. 
    
    Nevertheless, we emphasize that radial potentials are prevalent in many real-world applications across physics, biology, and materials science, radial potentials are ubiquitous. Examples include the well-known inverse power law \cite{Villani02} and Lennard-Jones potential \cite{PSJ_2022, PSJ_2023}. While non-radial interactions, such as those used to model flocking behavior \cite{motsch2011new}, have also been explored in the literature, these cases and interactions involving more than two particles will be addressed in future work.
\end{remark}

Finally, the dynamical system of OMD satisfied by the new variables $(x_i(t), w_i(t))$ is summarized as follows: for $i = 1,2,...,N$,
	\begin{equation}\label{simu_dynamic_xw}
		\left\{
		\begin{aligned}
			& \dot{x}_{i}(t) =
			w_{i}(t) + A(I+tA)^{-1}x_{i}(t),\\[6pt]
			& \dot{w}_{i}(t)
	    = - \sum_{\substack{j=1\\j \neq i}}^{N} \nabla_{x_i}U\left(\left|x_{i}(t) - x_{j}(t)\right|\right) - A(I+tA)^{-1} w_{i}(t).
		\end{aligned}
		\right.
	\end{equation}
 
\noindent\textbf{Result 1: From the microscopic regime to mesoscopic regime (Arrow (4))}\\
 
To obtain the corresponding kinetic equation from the fundamental dynamic system Eqs.~\eqref{simu_dynamic_xw}, it is crucial to apply an appropriate scaling operation. We follow the two classical scalings as follows: 
     \begin{itemize}
         \item \textbf{Mean-field type model}: it assumes that the contribution from each particle has the same weight $1/N$:
           \begin{equation}\label{simu_dynamic_xw_1N}
		\left\{
		\begin{aligned}
			& \dot{x}_{i}(t) =
			w_{i}(t) + A(I+tA)^{-1}x_{i}(t),\\[6pt]
			& \dot{w}_{i}(t)
	    = - \frac{1}{N}\sum_{\substack{j=1\\j \neq i}}^{N} \nabla_{x_i}U\left(\left|x_{i}(t) - x_{j}(t)\right|\right) - A(I+tA)^{-1} w_{i}(t)\,.
		\end{aligned}
		\right.
    \end{equation}
    
    By taking $N \rightarrow \infty$, we will obtain the mean-field type equation 
    \begin{multline}\label{P1UR}
        \frac{\partial g(t,x,w)}{\partial t} + w\cdot \nabla_{x}g(t,x,w) + [A(I+tA)^{-1} x] \cdot \nabla_{x} g(t,x,w) \\ -  [A(I+tA)^{-1} w] \cdot \nabla_{w} g(t,x,w) 
	=   \left[\nabla_{x}U*\rho_{g}\right](t,x) \cdot \nabla_{w}g(t,x,w),
    \end{multline}
    where $\rho_{g}(t,x) = \int_{\RR^{d}} g(t,x,w) \rd w$, and for the sake of rigorously justifying the limit, the potential $U \in C^1$ is assumed that $\nabla U$ is locally Lipschitz and $|\nabla U(x)| \leq C (1+|x|)$ for some constant $C>0$. Note that these requirements on $U$ are primarily for technical reasons, as they will be used  to prove the well-posedness of the mean-field type equation derived in Section \ref{sub:rigorous}.

    If further considering the homogeneity of $g$ in the space variable $x$ (cf.~ Eq.~\eqref{vw}), the Eq.~\eqref{P1UR} can be reduced to
    \begin{equation}\label{PR}
	\frac{\partial g(t,w)}{\partial t} - [A(I+tA)^{-1}w]\cdot \nabla_w g(t,w) = 0.
    \end{equation}

    The well-posedness of Eq.~\eqref{P1UR} and Eq.~\eqref{PR} has been established in Theorem~\ref{well-posed_P1UR}, and Theorem~\ref{conv} provides a rigorous connection between Eqs.~\eqref{simu_dynamic_xw_1N} and Eq.~\eqref{PR}.\\

    \item  \textbf{Boltzmann type model}: it emphasizes close neighbor interaction by rescaling the strength and range of the potential term from $\nabla U(r)$ to $\frac{1}{\varepsilon} \nabla U(\frac{r}{\varepsilon})$, i.e.,
    to derive the Boltzmann type model, for any given potential-induced force term $\nabla U(r)$, we rescale it by involving the parameter $\varepsilon$ in the following way: we add the factor $\frac{1}{\varepsilon}$ in the front to scale the strength and also add another $\frac{1}{\varepsilon}$ in $U$ to make it $\nabla U\left(\frac{r}{\varepsilon}\right)$ to scale the range of the potential:
    \begin{equation}\label{simu_dynamic_xwG}
	 \left\{
	 \begin{aligned}
	 	& \dot{x}_{i}(t) =
			w_{i}(t) + A(I+tA)^{-1}x_{i}(t),\\[6pt]
	 	& \dot{w}_{i}(t)
	     = - \frac{1}{\varepsilon} \sum_{\substack{j=1\\j \neq i}}^{N} \nabla_{x_i}U\left(\frac{|x_{i}(t) - x_{j}(t)|}{\varepsilon}\right) - A(I+tA)^{-1} w_{i}(t).
	 \end{aligned}
	 \right.
     \end{equation}

    By applying the Boltzmann-Grad limit, i.e., in the $d-$dimension, $(N-1) \varepsilon^{d-1} = O(1)  $ as $N \rightarrow \infty$ and $\varepsilon \rightarrow 0$, we have 
    \begin{equation}\label{BUR}
    \begin{split}
        \frac{\partial g(t,x,w)}{\partial t} + w \cdot \nabla_{x} g(t,&x,w) + [A(I+tA)^{-1} x] \cdot \nabla_{x} g(t,x,w) \\
        &-  [A(I+tA)^{-1} w] \cdot \nabla_{w} g(t,x,w) = \mathcal{Q}(g,g)(t,x,w)\,.
    \end{split}
    \end{equation}
    To achieve the desired form of collision operator $\mathcal{Q}(g,g)$ as in Eq.~\eqref{Q}, we assume that the potential $U \in C^2$ is a radial, non-negative, non-increasing function supported in a unit ball $\{ x \in \RR^d,\, 0<|x|<1 \}$ but unbounded near $|x| = 0$. Additionally, we require that both $U$ and $\nabla U$ vanish on the boundary of the unit ball, and satisfy the condition $|x|U''(|x|) + 2U'(|x|) \geq 0$ as specified in \cite[Assumption 1.2.1, Lemma 8.3.1]{GST2013}.
    The detailed derivation of Eq.~\eqref{BUR} from Eqs.~\eqref{simu_dynamic_xwG} via BBGKY hierarchy is laid out in Section~\ref{sub:formal-homo-boltzmann}, and the related properties are summarized in Section~\ref{sec:prop}.\\
    Similarly, if we further take the homogeneity of $g$ into account, Eq.~\eqref{BUR} becomes the so-called homo-energetic Boltzmann equation Eq.~\eqref{g}.
    \end{itemize}

It is worth noting that when the potential function $U(r)$ is a power function of $r$ (e.g., inverse power law $U(r) = \frac{1}{r^{\alpha}}$, where $\alpha > 1$ in $d=3$ dimension since $\alpha=1$ corresponds to the threshold case of Coulomb potential \cite{Villani02}), the two scaling strategies can be unified by extracting the scaling parameter $\ve$ out of $U(\frac{r}{\varepsilon})$. Additionally, our framework can also include another crucial scaling \cite[Eq.~(45)]{PS2021} that characterizes the weak interaction between molecules, specifically in the weak-coupling regime. This inclusion is anticipated to yield the homo-energetic Landau equation. For a detailed derivation of the Landau operator, we direct readers to \cite[Section 3.2]{PS2021}.\\

\noindent\textbf{Result 2: From the mesoscopic regime to macroscopic regime (Arrow (7))}\\

Another significant contribution of this paper is the derivation of the hydrodynamic equation from the kinetic equation, incorporating the structural properties inherited from OMD. Specifically, it bridges the gap identified as Arrow (7). As highlighted in \cite{James_ICM}, a specific family of unsteady macroscopic flows, associated with the simplest translation group \eqref{OMD-s}, inherently possesses a bulk velocity field $u(t,x) = A(I+tA)^{-1}x$ in Eulerian form. This velocity field naturally satisfies various steady and unsteady macroscopic fluid equations, leading us to anticipate that the conventional hydrodynamic systems governing the evolution of macroscopic quantities (density $\rho$ and temperature $\theta$ as defined in \eqref{rho} and \eqref{theta}, respectively) can be partially reduced.

Recall that $L(t) = A(I+tA)^{-1}$, we have:
\begin{itemize}
    \item By applying the Hilbert expansion, we derive a reduced Euler system from the homo-energetic Boltzmann equation Eq.~\eqref{g}:
    \begin{equation}\label{euler}
		\left\{
		\begin{aligned}
			\partial_{t}\rho(t) + \text{Tr}[L(t)] \rho(t)  &= 0, \\
			\partial_{t}\theta(t) + \frac{2}{3}\text{Tr}[L(t)]\theta(t) & = 0.
		\end{aligned}
		\right.
    \end{equation}
    Details are presented in Section~\ref{subsec:EL}.

    \item By applying the Chapman-Enskog expansion, we obtain the corresponding reduced Navier-Stokes system with $ O(\epsilon) $ correction terms from the homo-energetic Boltzmann equation Eq.~\eqref{g}:
    \begin{equation}\label{NS}
	\left\{
	\begin{aligned}
		\partial_{t}\rho(t) + \text{Tr}[L(t)] \rho(t)  &= 0, \\
		\partial_{t}\theta(t) + \frac{2}{3}\text{Tr}[L(t)]\theta(t) & = \epsilon \mu(\theta)\frac{1}{2}\Big( \text{Tr}[L^{2}(t)] + L(t):L(t) - \frac{2}{3} (\text{Tr}[L(t)])^2 \Big).
	\end{aligned}
	\right.
    \end{equation}
    where the viscosity $\mu$ is defined in \eqref{mu1}.
    See Section~\ref{subsec:NS} for more details.
\end{itemize}




\section{A Mean-field model for long-range interaction}
\label{sec:mean-field}

In this section, we focus on the derivation of a mesoscopic model from the mean-field scaling system described by Eqs.~\eqref{simu_dynamic_xw_1N}. This leads to the kinetic equation Eq.~\eqref{P1UR}, where the particle distribution function is influenced by an averaged force field. This equation can be further reduced to Eq.~\eqref{PR}. 

There are two approaches to complete the formalism of the mean-field limit on the single-particle phase space. One utilizes the concept of empirical measure and establishes the stability of the mean-field equation through Dobrushin's estimate \cite{Dobrushin, GMR2013}. The other approach, based on the BBGKY hierarchy, involves the $N$-particle distribution and demonstrates that it marginally satisfies the mean field equation. The former approach is simpler, while the latter is more flexible. In our presentation, we opt for the latter approach in the formal derivation as it can be applied to both scalings, and utilize the former approach for rigorous derivation. More details on the relation between these two approaches can be found in \cite{Golse2005, golse2016dynamics}.

\subsection{Derivation via BBGKY hierarchy}
\label{sub:derivation_mean_field}
Denote
\[
z_i : = (x_i, w_i)\,, \quad  Z_N = (z_1, \cdots, z_N) \in \Omega_N \,, \quad \Omega^{N} := \left\lbrace Z_N \in \RR^{6N} \ \big| \ x_{i} \neq  x_{j}, \ i\neq j \right\rbrace
\]
and let 
\[
P^{(N)}(t, Z_N) := P^{(N)}(t, z_{1},...,z_{N}) =P^{(N)} (t, x_{1}, w_{1},...,x_{N},w_{N})
\]
be the $N$-particle distribution function. Correspondingly, 
the $s-$marginal distribution  of $P^{(N)}$, denoted as $P^{(s)}(t, Z_s)$, is 
    \begin{equation} \label{0628}
        P^{(s)} (t, Z_s) := \int_{\RR^{6(N-s)}} P^{(N)}(t, Z_s, z_{s+1},...,z_{N}) \,\rd z_{s+1} ...\,\rd z_{N}\,, \quad
        Z_s = (z_1, z_2, \cdots, z_s)\,,
    \end{equation}
and then our goal is to derive the mean-field equation for the first marginal of the distribution $P^{(1)}(t, z_1)$.

Starting with the Liouville equation satisfied by $P^{(N)}(t,Z_N)$
	\begin{equation}\label{BBGKY0}
		\frac{\partial P^{(N)}(t,Z_N)}{\partial t} + \sum_{i=1}^{N} \left[ \dot{x}_i 
 \cdot \nabla_{x_i}  P^{(N)}  + \dot{w}_i \cdot  \nabla_{w_i}  P^{(N)} \right](t,Z_N) = 0\,,
	\end{equation}
and substituting Eqs.~\eqref{simu_dynamic_xw_1N} leads to 
    \begin{multline}\label{BBGKY00}
        \frac{\partial P^{(N)}}{\partial t} + \sum_{i=1}^N w_i\cdot \nabla_{x_i}P^{(N)} + \sum_{i=1}^N [A(I+tA)^{-1} x_i] \cdot \nabla_{x_i}P^{(N)} \\ 
        - \frac{1}{N} \sum_{i=1}^N \sum_{\substack{j=1\\j \neq i}}^{N} \nabla_{x_i}U\left(\left|x_{i} - x_{j}\right|\right) \cdot \nabla_{w_i}P^{(s)} 
        -\sum_{i=1}^N [A(I+tA)^{-1} w_i] \cdot \nabla_{w_i}P^{(N)} 
        = 0 \,.
    \end{multline}

Integrating Eq.~\eqref{BBGKY00} over the domain $\{z_{s+1},...,z_N\}$, we obtain the corresponding kinetic equation of the $s$-marginal distribution $P^{(s)}(t,Z_s)$,
    \begin{multline*}
        \frac{\partial P^{(s)}}{\partial t} + \underbrace{ \int_{\RR^{6(N-s)}} \left( \sum_{i=1}^N w_i\cdot \nabla_{x_i}P^{(N)} + \sum_{i=1}^N [A(I+tA)^{-1} x_i] \cdot \nabla_{x_i}P^{(N)} \right) \,\rd z_{s+1}...z_N }_{=:(I)} \\ 
        \underbrace{- \int_{\RR^{6(N-s)}} \left( \frac{1}{N} \sum_{i=1}^N \sum_{\substack{j=1\\j \neq i}}^{s} \nabla_{x_i}U\left(\left|x_{i} - x_{j}\right|\right) \cdot \nabla_{w_i}P^{(N)}-\sum_{i=1}^N [A(I+tA)^{-1} w_i] \cdot \nabla_{w_i}P^{(N)} \right) \,\rd z_{s+1}...z_N }_{=:(II)} \\
        = \underbrace{\int_{\RR^{6(N-s)}}  \frac{1}{N} \sum_{i=1}^N \sum_{\substack{j=s+1\\j \neq i}}^{N} \nabla_{x_i}U\left(\left|x_{i} - x_{j}\right|\right) \cdot \nabla_{w_i}P^{(N)}  \,\rd z_{s+1}...z_N }_{=:(III)}\,.
    \end{multline*}   

For term $(I)$, note that 
	\begin{equation*}
        \begin{split}
        (I) =& \sum_{i=1}^s w_i\cdot \nabla_{x_i}P^{(s)} + \sum_{i=1}^s \left[ A(I+tA)^{-1} x_i \right] \cdot \nabla_{x_i}P^{(s)} \\ 
         &\qquad \qquad + \sum_{i=s+1}^N \int_{\RR^{6(N-s)}} \left[ A(I+tA)^{-1} x_i \right] \cdot \nabla_{x_i}P^{(N)} \,\rd z_{s+1}...z_N \\
	=& \sum_{i=1}^s w_i\cdot \nabla_{x_i}P^{(s)} + \sum_{i=1}^s  [A(I+tA)^{-1} x_i] \cdot \nabla_{x_i} P^{(s)} - (N-s)\text{Tr}[A(I+tA)^{-1}] P^{(s)}\,.
        \end{split}
	\end{equation*}
	Similarly, term $(II)$ becomes
	\begin{equation*}
	    \begin{split}
	        (II) = &  - \frac{1}{N} \sum_{\substack{i,j=1\\i\neq j}}^{s} \nabla_{x_i}U\left(\left|x_{i} - x_{j}\right|\right)  \cdot \nabla_{w_i}P^{(s)} - \sum_{i=1}^s  [A(I+tA)^{-1} w_i] \cdot \nabla_{w_i} P^{(s)}   \\
	        & + (N-s) \text{Tr}[A(I+tA)^{-1}] P^{(s)}\,.
	    \end{split}
	\end{equation*}
Since particles are indistinguishable, term $(III)$ can be re-written as
	\begin{equation*}
	    \begin{split}
	        (III) = & \frac{N-s}{N} \sum_{i=1}^s \int_{\RR^{6}} \nabla_{x_i}U\left(\left|x_{i} - x_{s+1}\right|\right) \cdot \nabla_{w_i}P^{(s+1)}(t,Z_s,z_{s+1})\,\rd z_{s+1}\\
	        = & \frac{N-s}{N} \sum_{i=1}^s \nabla_{w_i} \cdot \int_{\RR^{6}} \left[  \nabla_{x_i}U\left(\left|x_{i} - x_{s+1}\right|\right)  P^{(s+1)}(t,Z_s,z_{s+1}) \right] \,\rd z_{s+1} \,.
	    \end{split}
	\end{equation*}
	
Combining the terms $(I)-(III)$ altogether, we arrive at the following equation for the marginal distribution $P^{(s)}$ 
	\begin{multline}\label{P1U}
	     \frac{\partial P^{(s)}}{\partial t} + \sum_{i=1}^s w_i\cdot \nabla_{x_i}P^{(s)} + \sum_{i=1}^s  [A(I+tA)^{-1} x_i] \cdot \nabla_{x_i} P^{(s)}\\
      - \sum_{i=1}^s  [A(I+tA)^{-1} w_i] \cdot \nabla_{w_i} P^{(s)} - \frac{1}{N}\sum_{\substack{i,j=1\\i\neq j}}^{s} \nabla_{x_i}U\left(\left|x_{i} - x_{j}\right|\right) \cdot \nabla_{w_i}P^{(s)}\\
      = \frac{N-s}{N} \sum_{i=1}^s \nabla_{w_i} \cdot \int_{\RR^{6}} \left[  \nabla_{x_i}U\left(\left|x_{i} - x_{s+1}\right|\right)  P^{(s+1)}(t,Z_s,z_{s+1}) \right] \,\rd z_{s+1} \,.
	\end{multline}
In particular, taking $s=1$ in Eq.~\eqref{P1U} above, it reduces to the two-particle case:
\begin{equation}\label{P2U}
	\begin{split}
	    \frac{\partial P^{(1)}}{\partial t}& + w_1\cdot \nabla_{x_1}P^{(1)} + [A(I+tA)^{-1} x_1] \cdot \nabla_{x_1} P^{(1)} -  [A(I+tA)^{-1} w_1] \cdot \nabla_{w_1} P^{(1)} \\
	    = &  \frac{N-s}{N} \nabla_{w_1} \cdot \int_{\RR^{6}} \left[  \nabla_{x_1}U\left(\left|x_{1} - x_{2}\right|\right)  P^{(2)}(t,z_1,z_2) \right] \,\rd z_{2}\,. 
	\end{split}
\end{equation}

To close the hierarchy above, we consider the following ``propagation of chaos'' assumption \cite{Villani02}: 
	\begin{equation*}
	    P^{(2)}(t,z_1,z_2) = P^{(1)}(t,x_1,w_1) P^{(1)}(t,x_2,w_2)\,,
	\end{equation*}
which says the two particles remain independent throughout the dynamics. Under this assumption, the right-hand side of Eq.~\eqref{P2U} becomes
\begin{equation}\label{G0}
\begin{split}
    &\frac{N-1}{N} \nabla_{w_1} \cdot \int_{\RR^{6}} \left[  \nabla_{x_1}U\left(\left|x_{1} - x_{2}\right|\right)  P^{(2)}(t,z_1,z_2) \right] \,\rd z_{2}\\
    =& \frac{N-1}{N}  \int_{\RR^{6}} \left[  \nabla_{x_1}U\left(\left|x_{1} - x_{2}\right|\right)  P^{(1)}(t,x_2,w_2) \nabla_{w_1}P^{(1)}(t,x_1,w_1)  \right] \,\rd x_{2} \,\rd w_2\\
    =& \frac{N-1}{N} \int_{\RR^{3}} \left[  \nabla_{x_1}U\left(\left|x_{1} - x_{2}\right|\right) \int_{\RR^{3}} P^{(1)}(t,x_2,w_2)\,\rd w_2  \right] \,\rd x_{2} \cdot \nabla_{w_1}P^{(1)}(t,x_1,w_1)\\
    =& \frac{N-1}{N} \nabla_{x_1}U*\rho_{P^{(1)}}(t,x_1) \cdot \nabla_{w_1}P^{(1)}(t,x_1,w_1).
\end{split}
\end{equation}
By sending $N \rightarrow \infty$ and re-naming $P^{(1)}(t,x_1,w_1)$ by $g(t,x,w)$, the Eq.~\eqref{P2U} is actually Eq.~\eqref{P1UR}. Furthermore, since molecules in different $x$ see the same environment,
the spatial dependence is removable at the kinetic level. Therefore, $g$ is a spatially homogeneous function, which then obeys the reduced mean-field equation Eq.~\eqref{PR}.



\subsection{Rigorous justification of the mean-field equation}
\label{sub:rigorous}

In this subsection, we underpin the well-posedness of Eq.~\eqref{P1UR}, and establish a rigorous path from OMD to the mean field equation. Our approach will follow that in \cite{CCR2011}.

First, we set up some notation. We denote $\mathcal{P}_1(\mathbb{R}^3 \times \mathbb{R}^3)$ as the space of probability measures on $\mathbb{R}^3 \times \mathbb{R}^3$ with a finite first moment. This space is equipped with the Monge-Kantorovich-Rubinstein distance $W_1$, defined as: for $V = (X, W) \in \RR^3 \times \RR^3$,
    \begin{equation*}
    W_1(\mu,\nu) := \sup \Big\{ \Big| \int_{\RR^3 \times \RR^3} \varphi(V) (\mu(V) - \nu(V)) \,\rd V  \Big|, \ \varphi\in \text{Lip}(\RR^3 \times \RR^3),\ \|\varphi\|_{\rm{Lip}}\leq 1 \Big\},
    \end{equation*}
where $\text{Lip}(\mathbb{R}^3 \times \mathbb{R}^3)$ denotes the set of Lipschitz functions on $\mathbb{R}^3 \times \mathbb{R}^3$, and $\|\cdot\|_{\text{Lip}}$ represents the associated norm. Additionally, we define $\mathcal{P}_c(\mathbb{R}^3 \times \mathbb{R}^3)$ as a subset of $\mathcal{P}_1(\mathbb{R}^3 \times \mathbb{R}^3)$ with compact support. We also introduce a metric space $\mathcal{G} := C\big([0,T], \mathcal{P}_c(\mathbb{R}^3 \times \mathbb{R}^3)\big)$ associated with the distance $\mathcal{W}_1$ defined as follows: for $g_t(V):=g(t,V) $ and $ h_t(V):=h(t,V)$ in $\mathcal{G}$,
\begin{align}\label{def0622}
    \mathcal W_1(g(\cdot, \cdot), h(\cdot, \cdot)) := \sup_{t\in [0,T]} W_1(g_t(\cdot), h_t(\cdot)). 
\end{align}

Compared to the classical mean-field equation in \cite{CCR2011}, the essential difference of Eq.~\eqref{P1UR} lies in the left-hand side, where the characteristic trajectory $(X, W):=(X(t), W(t))$ is written as follows
\begin{equation}\label{XW}
		\left\{
		\begin{aligned}
			& \frac{\rd}{\rd t} X =
			W + A(I+tA)^{-1}X,\\[6pt]
			& \frac{\rd}{\rd t} W
	    = -\nabla U*\rho_g(t,X) - A(I+tA)^{-1} W\,.
		\end{aligned}
		\right.
\end{equation}
In the rest of this subsection, we will take the simple shear as an example (see \cite[Theorem 3.1]{JNV2019}), in which case $A$ is rank-1 and 
traceless, i.e., 
\begin{equation}\label{simple_shear}
    L(t) = A(I+tA)^{-1} = 
    \begin{pmatrix}
    0 & K & 0\\
    0 & 0 & 0 \\
    0 & 0 & 0
\end{pmatrix}, 
\quad \text{with} \ K \neq 0.
\end{equation}

In fact, for the purpose of future extension, we consider a rather general field $\xi$ and an operator $\mathcal{H}$ that satisfy a certain class of hypotheses, instead of assuming a specific form. Specifically, we consider the following system:
\begin{equation}\label{XWG}
		\left\{
		\begin{aligned}
			& \frac{\rd}{\rd t} X =
			\xi(t,X,W),\\[6pt]
			& \frac{\rd}{\rd t} W = \mathcal{H}[g](t,X,W) = E[g](t,X) + \eta(t,W)\,.
		\end{aligned}
		\right.
\end{equation}

Here, we will present the sufficient hypotheses that ensure the fulfillment of $\xi$ and $\mathcal{H}$ for the specific case of Eq.~\eqref{XW}. These hypotheses guarantee the well-posedness of the mean-field equation Eq.\eqref{P1UR}.

\begin{hypothesis}\label{Hypo_xi}[Hypotheses on $\xi$]

\noindent (i) $\xi(t,x,w)$ is continuous on $[0,T] \times \RR^3 \times \RR^3$.\\
(ii) There exists a constant $C_{\xi} > 0$,
\begin{equation}\label{xi_growth}
    |\xi(t,x,w)| \leq C_{\xi} (1 + |x| + |w|), \quad \forall t,x,w \in [0,T] \times \RR^3 \times \RR^3\,.
\end{equation}
(iii) $\xi$ is locally Lipschitz in variables $x$ and $w$, i.e., for any compact set $D \subset \RR^3 \times \RR^3$, there is a constant $L_{\xi} = L_{\xi}(D) > 0$ such that,
\begin{equation*}
    |\xi(t, V_1) - \xi(t,V_2)| \leq L_{\xi} |V_1 - V_2|,\quad t \in [0,T], \quad V_1, V_2 \in D.
\end{equation*}
\end{hypothesis}

\begin{remark}
    Note that $\xi(t,x,w) = w + A(I+tA)^{-1}x $ satisfies the Hypothesis \ref{Hypo_xi} with $A(I+tA)^{-1} $ being a simple shear as in \eqref{simple_shear}. Indeed, we have, for all $ t \in [0,T]$, 
    \begin{equation*}
    \begin{split}
        |\xi(t,x,w)| = |w + A(I+tA)^{-1}x| \leq &\ |w| + K|x|\\
        \leq & \ C_{\xi}(1 + |w| + |x|),
    \end{split}
    \end{equation*}
    where $C_{\xi} = 1+K$. On the other hand, for all $t \in [0,T]$ and $ V_1, V_2 \in D$,
    \begin{equation*}
    \begin{split}
        |\xi(t, V_1) - \xi(t,V_2)| \leq & \ |w_1 - w_2| + K |x_1 - x_2|\\
        \leq &\ L_{\xi} |V_1 - V_2|\,,
    \end{split}
    \end{equation*}
    where $L_{\xi} = 2(1+K)$.
\end{remark}


\begin{hypothesis}\label{Hypo_H}[Hypotheses on $\mathcal{H}$]

\noindent (i) $\mathcal{H}[g](t,x,w)$ is continuous on $[0,T] \times \RR^3 \times \RR^3$.\\
(ii) For any $g(t,\cdot,\cdot) \in P_{c}(\RR^3 \times \RR^3)$ with support contained in a ball $B_R \subset \RR^3 \times \RR^3$ and  for all $t \in [0,T]$, there exists a constant $C_{\mathcal{H}} = C_{\mathcal{H}}(R, T) > 0$,
\begin{equation}\label{H_growth}
    \|\mathcal{H}[g](t,\cdot,\cdot)\|_{L^{\infty}(B_R)} \leq C_{\mathcal{H}}, \quad \forall t \in [0,T].
\end{equation}
(iii) For $g, h \in P_{1}(\RR^3 \times \RR^3)$ and any ball $B_R \subset \RR^3 \times \RR^3$,
\begin{equation}\label{H-hypo31}
    \|\mathcal{H}[g](\cdot,\cdot) - \mathcal{H}[h](\cdot,\cdot)\|_{L^{\infty}(B_R)} \leq {\rm{Lip}}_{R}\big[\mathcal{H}(\cdot,\cdot)\big] W_1\big(g(\cdot,\cdot),h(\cdot,\cdot)\big).
\end{equation}
Furthermore, if $g_t,h_t \in \mathcal{G}$ such that ${\rm{supp}}(g_t) \cup {\rm{supp}}(h_t) \subset B_{R_0}$ for all $t \in [0,T]$. Then for any ball $B_R \subset \RR^3 \times \RR^3$, there exists a constant $L_{\mathcal{H}} = L_{\mathcal{H}}(R,R_0)$ such that
\begin{equation}\label{H-hypo32}
    \max_{t \in [0,T]} \|\mathcal{H}[g](t,\cdot,\cdot) - \mathcal{H}[h](t,\cdot,\cdot)\|_{L^{\infty}(B_R)} \leq L_{\mathcal{H}} \mathcal{W}_1\big(g(\cdot,\cdot,\cdot),h(\cdot,\cdot,\cdot)\big),
\end{equation}
with
\begin{equation*}
    \max_{t \in [0,T]} {\rm{Lip}}_{R}\big[\mathcal{H}(t,\cdot,\cdot)\big] \leq L_{\mathcal{H}}.
\end{equation*}

\end{hypothesis}

\begin{remark}
    It can be illustrated that the particular operator $\mathcal{H}[g](t,x,w) = E[g](t,x) + \eta(t,w)$ in the designated model Eqs.~\eqref{XWG} satisfies the desired Hypothesis \ref{Hypo_H}, as long as $E[g](t,x)= -\nabla U*\rho_g(t,x)$ satisfies the Hypothesis~\ref{Hypo_E} as in Appendix~\ref{app:hypo}. Consequently, the Lipschitz constant ${\rm{Lip}}_{R}\big[\mathcal{H}(\cdot,\cdot)\big]$ in \eqref{H-hypo31} and $\max_{t \in [0,T]} {\rm{Lip}}_{R}\big[\mathcal{H}(t,\cdot,\cdot)\big]$ in \eqref{H-hypo32} will depend on the potential $U$.
\end{remark}

Finally, we can define the flow operator at time $t \in [0,T)$ of Eqs.~\eqref{XWG},
\begin{equation*}
    \mathcal{T}_{\xi,\mathcal{H}}^t: (X(0), W(0)) \in \RR^3 \times \RR^3 \mapsto  (X(t), W(t)) \in \RR^3 \times \RR^3\,.
\end{equation*}
Following the definition of the solution as in \cite[Definition 3.3]{CCR2011}, for an initial probability measure $g_0(x,w) \in \mathcal{P}_1(\RR^3 \times \RR^3)$, the function 
\begin{equation}\label{DefT}
    g(t,x,w): [0,T) \rightarrow \mathcal{P}_1(\RR^3 \times \RR^3), \quad t \mapsto g_t(x,w) := \mathcal{T}_{\xi,\mathcal{H}}^t \#g_0(x,w)
\end{equation}
is a measure-valued solution to Eq.~\eqref{P1UR} in the distributional sense, where $g(t,x,w)= g_t(x,w)= \mathcal{T}_{\xi,\mathcal{H}}^t \#g_0(x,w)$ is defined as
\begin{equation*}
    \int_{\RR^3 \times \RR^3} \zeta(x,w) g(t,x,w) \,\rd x \,\rd w = \int_{\RR^3 \times \RR^3} \zeta\left(\mathcal{T}_{\xi,\mathcal{H}}^t(x,w)\right) g_0(x,w) \,\rd x \,\rd w
\end{equation*}
for all $\zeta \in C_b(\RR^3 \times \RR^3)$.




\subsubsection{Well-posedness theorem of mean-field equation}

  Our main well-posedness theorem for Eq.~\eqref{P1UR} states as follows: 

    \begin{theorem}[Existence, uniqueness and stability]\label{well-posed_P1UR}

    Assume that the field $\xi(t,x,w)$ satisfies the Hypothesis~\ref{Hypo_xi} and the operator $\mathcal{H}(t,x,w)$ satisfies the Hypothesis~\ref{Hypo_H}.
    
    For any initial datum $g_0(x,w) \in \mathcal{P}_c(\RR^3 \times \RR^3)$, there exists a measure-valued solution $g_t(x,w) = g(t,x,w) \in C\big([0,+\infty), \mathcal{P}_c(\RR^3 \times \RR^3)\big)$ to Eq.~\eqref{P1UR}, and there is an increasing function $R = R(T)$ such that for all $T>0$,
	\begin{equation}\label{gincrease}
	    {\rm{supp}} \ g_t(\cdot,\cdot) \subset B_{R(T)} \subset \mathbb{R}^3 \times \mathbb{R}^3, \quad \forall \ t \in [0,T].
	\end{equation}
    This solution is unique among the family of solutions $C\big([0,+\infty), \mathcal{P}_c(\mathbb{R}^3 \times \mathbb{R}^3)\big)$ satisfying \eqref{gincrease}.
    
    Moreover, the solution depends continuously with respect to the initial data in the following sense. Assume that $g_0, h_0 \in \mathcal{P}_c(\mathbb{R}^3 \times \mathbb{R}^3)$ are two initial conditions, and $g_t,h_t$ are the corresponding solutions to Eq.~\eqref{P1UR}. Then,
    \begin{equation*}
    W_1(g_t(\cdot,\cdot), h_t(\cdot,\cdot)) \leq  \e^{2tL} W_1(g_0(\cdot,\cdot),h_0(\cdot,\cdot)), \quad \forall t \geq 0,
    \end{equation*}
    where $L = \max\{L_{V}, L_{\mathcal{H}} \}$ with $L_{V}$ in Lemma \ref{RCE}.
    
\end{theorem}

\begin{proof}
    \textbf{(Existence and uniqueness):} Given any initial condition $g_0(x,w) \in \mathcal{P}_c(\RR^3 \times \RR^3)$ with support contained in a ball $B_{R_0} \subset \RR^3 \times \RR^3$ for some $R_0 > 0$, we prove the local existence and uniqueness of solutions by a fixed-point argument in a complete metric space $(\mathcal{G}, \mathcal{W}_1)$ defined in \eqref{def0622}, where the support of $g(t,x,w)$ is contained in $B_R$ for all $ t \in [0,T]$ with $R = 2R_0$, and $T > 0$ is fixed time that will be determined later on. 

   We now define an operator $\Gamma$ on the space $\mathcal{G}$ such that its fixed point is the solution to the mean-field equation Eq.~\eqref{P1UR}. For $g \in \mathcal{G}$, if the field $\xi(t,x,w)$ and operator $\mathcal{H}[\cdot](t,x,w)$ satisfy the Hypothesis~\ref{Hypo_xi} and ~\ref{Hypo_H}, respectively, we define:
    \begin{equation*}
        \Gamma[g](t,x,w) := \mathcal{T}^t_{\xi,\mathcal{H}[g]} \#g_0(x,w)\,.
    \end{equation*}
    Clearly, if $g$ is the solution to Eq.~\eqref{P1UR} with the initial condition $g_0(x,w)$, then $\Gamma[g]$ also solves the same initial value problem. This can be demonstrated using the method of characteristics.

    To invoke the fixed-point argument, we need to accomplish the following two tasks. 
    \begin{itemize}
    \item[(I)] Show that $\Gamma$ maps $g \in \mathcal{G}$ to the same space $\mathcal{G}$ under an appropriate choice of time $T_1$.That is, we need to show that $\mathcal{T}^t_{\xi,\mathcal{H}[g]} \#g_0(x,w)$ is a probability measure in $\mathcal{P}_1$ with compact support in $B_R$.
    
    Thanks to the Hypothesis~\ref{Hypo_xi} and~\ref{Hypo_H}  on $\xi(t,x,w)$ and $\mathcal{H}(t,x,w) $, using Lemma \ref{RCE}, we see that $|\frac{\rd }{\rd t} \mathcal{T}_{\xi,\mathcal{H}[g]}(V)| \leq C_V $ for all $V \in B_{R_0} \subset \RR^3 \times \RR^3$ with $C_V>0$ depending on $R_0$, $T$, $C_{\xi}$ and $C_{\mathcal{H}}$. Then as long as $T_1$ is selected such that $T_1 \leq \frac{R_0}{C_P}$, the support of $\Gamma[g](t,x,w) = \mathcal{T}^t_{\xi,\mathcal{H}[g]} \#g_0(x,w)$ is contained in $B_R$ with $R = 2R_0$. Meanwhile, $\Gamma[g](t,x,w) \in \mathcal{P}_1(\RR^3 \times \RR^3)$ is fulfilled by the mass conservation. Moreover, the time continuity of $\Gamma[g]$, i.e., $t \mapsto \Gamma[g](t,\cdot,\cdot)$ follows from Lemma \ref{con_t_g}.  
    This implies that the operator $\Gamma[g]$ mapping from the space $\mathcal{G}$ into itself is well-defined.\\

    \item[(II)] Demonstrate that $\Gamma$ is a contraction map in $\mathcal{G}$ for specific choices of $T_2$. This is to show that, for any $g, h \in \mathcal{G}$, the following inequality holds:
   \begin{equation*}
        \mathcal{W}_1 \big(\Gamma[g](\cdot,\cdot,\cdot), \Gamma[h](\cdot,\cdot,\cdot) \big) \leq C \mathcal{W}_1 \big(g(\cdot,\cdot,\cdot), h(\cdot,\cdot,\cdot)\big)\,,
    \end{equation*}
     where $0 < C < 1$ is a constant independent of $g$ and $h$. Note that, starting from the same $g_0 \in \mathcal{P}_c(\RR^3 \times \RR^3)$ with support contained in $B_R$, we have
    \begin{equation}\label{W1sup}
        \mathcal{W}_1 \big(\Gamma[g](\cdot,\cdot,\cdot), \Gamma[h](\cdot,\cdot,\cdot)\big) = \sup_{t \in [0,T_2]} W_1 \Big(\mathcal{T}^t_{\xi,\mathcal{H}[g]}\# g_0(\cdot,\cdot), \mathcal{T}^t_{\xi,\mathcal{H}[h]}\#g_0(\cdot,\cdot) \Big)
    \end{equation}
    

    
    Further, for $t \in [0,T_2]$, we have 
    \begin{equation}\label{WW}
        \begin{split}
            & W_1 \Big(\mathcal{T}^t_{\xi,\mathcal{H}[g]}\#g_0(\cdot,\cdot), \mathcal{T}^t_{\xi,\mathcal{H}[h]}\#g_0(\cdot,\cdot) \Big) 
            \\ \leq & \frac{\e^{tL_V}-1}{L_V} \left(  \sup_{\tau \in (0,T)} \|\mathcal{H}[g](\tau, \cdot,\cdot)  - \mathcal{H}[h](\tau, \cdot,\cdot) \|_{L^\infty(\text{supp}g_0)} \right)\\
            \leq & \frac{\e^{tL_V}-1}{L_V} L_{\mathcal{H}} \sup_{\tau \in [0,T_2]} W_1(g_\tau(\cdot,\cdot) - h_\tau(\cdot,\cdot)) \\
            = & \frac{\e^{tL_V}-1}{L_V} L_{\mathcal{H}} \mathcal{W}_1(g(\cdot,\cdot,\cdot),h(\cdot,\cdot,\cdot))
        \end{split}
    \end{equation}
    where the first inequality comes from Lemma \ref{DCE} when $\xi_1=\xi_2 = \xi$, while the second inequality utilizes Hypothesis \ref{Hypo_H}. 
   
    Taking the maximum over $t \in [0,T_2]$ in \eqref{WW}, we see that
    \begin{equation*}
         \mathcal{W}_1 \big(\Gamma[g](\cdot,\cdot,\cdot), \Gamma[h](\cdot,\cdot,\cdot)\big) \leq   \frac{\e^{T_2L_V}-1}{L_V} L_{\mathcal{H}} \mathcal{W}_1 \big(g(\cdot,\cdot,\cdot),h(\cdot,\cdot,\cdot)\big)\,.
    \end{equation*}
    Since $\lim_{T_2 \rightarrow 0 } \frac{\e^{T_2L_V}-1}{L_V} = 0 $, 
    we can choose $T_2$ small enough such that $\frac{\e^{T_2L_V}-1}{L_V} L_{\mathcal{H}} < 1$. This ensures that the mapping $\Gamma$ is contractive on $\mathcal{G}$.
\end{itemize}

    Combining the analysis above, we prove the existence of a unique fixed point of $\Gamma[g]$ in $\mathcal{G}$ by selecting $T = \min\{T_1, T_2\}$. This fixed point, denoted as $g(t,x,w)$, represents the unique solution to Eq.~\eqref{P1UR} within the local time interval $[0,T]$.

    Moreover, since the time $T$ is independent of the initial condition and mass conservation is ensured, we can extend the solution equally to any global time interval by repeating the same argument. This extension is valid as long as the support of the solution remains compact, which has been verified in Lemma \ref{XtWt}.
    
    \textbf{(Stability):} Following the previous argument, we can choose any fixed $T > 0$ and $R > 0$ such that the supports of $g_t(x,w)$ and $h_t(x,w)$ are contained in $B_R$ for $t \in [0,T]$. Then we have: 
    \begin{equation*}
        \begin{split}
            & W_1 (g_t(\cdot, \cdot), h_t(\cdot, \cdot)) 
            \\ = & W_1 \Big( \mathcal{T}^t_{\xi,\mathcal{H}[g]} \#g_0(\cdot,\cdot), \mathcal{T}^t_{\xi,\mathcal{H}[h]}\#h_0(\cdot,\cdot) \Big) \\[2pt]
            \leq & W_1 \Big( \mathcal{T}^t_{\xi,\mathcal{H}[g]} \#g_0(\cdot,\cdot), \mathcal{T}^t_{\xi,\mathcal{H}[h]} \#g_0(\cdot,\cdot) \Big) + W_1 \Big( \mathcal{T}^t_{\xi,\mathcal{H}[h]} \#g_0(\cdot,\cdot), \mathcal{T}^t_{\xi,\mathcal{H}[h]} \#h_0(\cdot,\cdot) \Big) \\[4pt]
            \leq & \| \mathcal{T}^t_{\xi,\mathcal{H}[g]} - \mathcal{T}^t_{\xi,\mathcal{H}[h]} \|_{L^{\infty}(\text{supp}\ g_0)} + {\rm{Lip}}_R\big[\mathcal{T}^t_{\xi,\mathcal{H}[h]}\big] W_1\big(g_0(\cdot,\cdot), h_0(\cdot,\cdot)\big) \\[2pt]
            \leq & \int_0^t \e^{(t-\tau)L_V}\|\mathcal{H}[g](\tau, \cdot,\cdot)  - \mathcal{H}[h](\tau, \cdot,\cdot) \|_{L^\infty(B_R)} \,\rd\tau + \e^{tL_V} W_1\big(g_0(\cdot,\cdot), h_0(\cdot,\cdot)\big) \\
            \leq & \int_0^t \e^{(t-\tau)L_V}  {\rm{Lip}}_{R}\big[\mathcal{H}\big] W_1\big(g_{\tau}(\cdot,\cdot),h_{\tau}(\cdot,\cdot)\big) \,\rd\tau + \e^{tL_V} W_1\big(g_0(\cdot,\cdot), h_0(\cdot,\cdot)\big) \,.
        \end{split}
    \end{equation*}
    Note that ${\rm{Lip}}_{R}\big[\mathcal{H}\big] \leq L_{\mathcal{H}}$ for all $t \in [0,T]$, we can choose $L = \max\{ L_V, L_{\mathcal{H}} \}$ such that, for all $t \in [0,T]$
    \begin{equation*}
        \e^{-tL} W_1 (g_t(\cdot, \cdot), h_t(\cdot, \cdot)) \leq L\int_0^t \e^{-\tau L_V} W_1\big(g_{\tau}(\cdot,\cdot),h_{\tau}(\cdot,\cdot)\big) \,\rd\tau + W_1\big(g_0(\cdot,\cdot)-h_0(\cdot,\cdot)\big)\,.
    \end{equation*}
    Then using Gronwall's inequality, we have the following estimate
    \begin{equation*}
        W_1 (g_t(\cdot, \cdot), h_t(\cdot, \cdot)) \leq \e^{2tL} W_1\big(g_0(\cdot,\cdot)-h_0(\cdot,\cdot)\big), \quad \forall t \in [0,T].
    \end{equation*}
    This completes the proof.

\end{proof}


\subsubsection{Proof of the mean-field limit}

As a consequence of the well-posedness established in Theorem \ref{well-posed_P1UR}, we are able to offer a theoretical justification for the mean-field limit, i.e., Eq.~\eqref{P1UR}. Apart from the derivation via the BBGKY hierarchy, an alternative approach to obtaining the mean field equation is by assuming that the solution $g$ represents an empirical measure of a collection of particles, characterized as follows.

\begin{lemma}
    Consider the following dynamical system 
    \begin{equation}\label{XWG_small}
		\left\{
		\begin{aligned}
			& \frac{\rd}{\rd t} x_i =
			\xi(t,x_i,w_i), \quad i = 1,...,N,\\[6pt]
			& \frac{\rd}{\rd t} w_i = \mathcal{H}[g^N_t](t,x_i,w_i), \quad i = 1,...,N,
		\end{aligned}
		\right.
    \end{equation}
    where $\xi(t,x,w)$ and $\mathcal{H}(t,x,w)$ satisfy the Hypothesis~\ref{Hypo_xi} and ~\ref{Hypo_H}, respectively. Let $g^N_t(x,w): [0,T] \mapsto \mathcal{P}_1(\RR^3 \times \RR^3)$ be a probability measure defined as
    \begin{equation}\label{delta_g_t}
        g_t^{N}(x,w) := \frac{1}{N} \sum_{i=1}^{N} \delta(x-x_{i}(t)) \delta(w-w_{i}(t)) .
    \end{equation}
    If $x_i, w_i: [0,T] \mapsto \RR^3$, for $i = 1,...,N$, is a solution to Eqs.~\eqref{XWG_small}, then $g_t^N(x,w)$ is the measure-valued solution to Eq.~\eqref{P1UR} with the initial condition
    \begin{equation}\label{delta_g_initial}
        g_0^N(x,w) := \frac{1}{N} \sum_{i=1}^{N} \delta(x-x_{i}(0)) \delta(w-w_{i}(0)) .
    \end{equation}
\end{lemma}

For the sake of completeness, the detailed derivation is provided in Appendix \ref{app:delta}. With the support of Theorem \ref{well-posed_P1UR}, we can rigorously justify the convergence of the measure-valued solution. 

\begin{theorem}[Convergence of the empirical measure] \label{conv}
   Under the Hypothesis~\ref{Hypo_xi} and~\ref{Hypo_H}, 
    for any initial datum $g_0 \in \mathcal{P}_c(\mathbb{R}^3 \times \mathbb{R}^3)$, consider a sequence of $g_0^N$ in the form of \eqref{delta_g_initial}  such that
    \begin{equation*}
        \lim_{N\rightarrow \infty} \mathcal{W}_1 \big(g_0^N(\cdot,\cdot), g_0(\cdot,\cdot)\big) = 0.
    \end{equation*}
    Let $g_t^N$ be given by \eqref{delta_g_t}, where $(x_i(t), w_i(t))$ solves Eqs.~\eqref{simu_dynamic_xw_1N} with initial conditions $(x_i(0), w_i(0))$. Then we have
    \begin{equation*}
        \lim_{N\rightarrow \infty} \mathcal{W}_1 \big(g_t^N(\cdot,\cdot), g_t(\cdot,\cdot)\big) = 0
    \end{equation*}
    for all $t \geq 0$, where $g_t(x,w)$ is the unique measure-valued solution to Eq.~\eqref{P1UR} with initial data $g_0(x,w)$.
\end{theorem}

\section{A Boltzmann-type model for short-range interaction}
\label{sec:boltzmann}
	
In this section, we will derive the Boltzmann type equation based on the scaling in Eqs.~\eqref{simu_dynamic_xwG}.

	

\subsection{Derivation via BBGKY Hierarchy}
\label{sub:formal-homo-boltzmann}

As in the previous section, we begin with the Liouville equation Eq.~\eqref{BBGKY0}. However, in this case, we need to account for the range of interaction and define the marginals in the truncated domain $\RR^3 \setminus \{ Z_N,\ |x_i-x_j| \leq \ve, \ \text{for}\ i\neq j\}$. The marginals are now denoted as $\tilde{P}^{(s)} (t, Z_s)$ instead of $P^{(s)} (t, Z_s)$: 


\begin{equation*}
\begin{split}
    \tilde{P}^{(s)} (t, Z_s) := & \int_{\RR^{6(N-s)}} P^{(N)}(t, z_{1}, z_{2},...,z_{N}) \prod_{\substack{i\in [1,s] \\ j\in [s+1,N]}} \mathrm{1}_{|x_i-x_j|>\ve} \,\rd z_{s+1} ...\,\rd z_{N}\\
    = & \int_{\RR^{6(N-s)}} P^{(N)}(t, z_{1}, z_{2},...,z_{N})  \mathrm{1}_{X_N \in \mathcal{D}_N^s} \,\rd z_{s+1} ...\,\rd z_{N}\,,
\end{split}
\end{equation*}
where
\begin{equation*}
    \mathcal{D}_{N}^s:= \Big\{ (x_1,...,x_N) \in \RR^{3N} \Big| \ |x_i - x_j| > \ve,\ \forall (i,j) \in [1,s] \times [s+1,N] \Big\}\,,
\end{equation*}
and $Z_s$ is defined in \eqref{0628}. 
Our derivation follows the strategies outlined in King \cite{King1975} and Gallagher et al.\cite{GST2013}. When $s=1$, the truncated one-particle marginal $\tilde{P}^{(1)}$ is 
\begin{equation}\label{tP}
\begin{split}
    \tilde{P}^{(1)} (t, z_1) =& \int_{\RR^{6(N-1)}} P^{(N)}(t, z_{1}, z_{2},...,z_{N}) \prod_{j\geq 2}^{N} \mathrm{1}_{|x_1-x_j|>\ve} \,\rd z_{2} ...\,\rd z_{N}\\
    =& \int_{\RR^{6(N-1)}} P^{(N)}(t, z_{1}, z_{2},...,z_{N})  \mathrm{1}_{X_N \in \mathcal{D}_N^1} \,\rd z_{2} ...\,\rd z_{N}\,.
\end{split}
\end{equation}
Then our goal is to find the weak form satisfied by $\tilde{P}^{(1)}(t,z_1)$.  

To this end, we first derive a more general form satisfied by $s$-particle marginals $\tilde{P}^{(s)}(t, Z_s)$. Given a smooth and compactly supported function $\phi(t, Z_s)$ defined on $\RR_{+} \times \RR^{6s}$, we have, by considering Eqs.~\eqref{simu_dynamic_xwG} and starting from Eq.~\eqref{BBGKY0},
\begin{equation}\label{PN1}
\begin{split}
        \int_{\RR_{+} \times \RR^{6N}}& \Big( \frac{\partial P^{(N)}(t,Z_N)}{\partial t} + \sum_{i=1}^{N}  \left[ w_i\cdot \nabla_{x_i}  P^{(N)} + [A(I+tA)^{-1} x_i] \cdot \nabla_{x_i}P^{(N)} \right](t,Z_N)\\
		-  \sum_{i=1}^{N} & \Big[ \frac{1}{\ve}\sum_{\substack{j=1\\i\neq j}}^{N} G\left(\frac{|x_{i} - x_{j}|}{\varepsilon}\right)\cdot \nabla_{w_i} P^{(N)} + [A(I+tA)^{-1} w_i] \cdot \nabla_{w_i}P^{(N)} \Big](t,Z_N) \Big)\\ 
		&\times \phi(t,Z_s) \mathrm{1}_{X_N \in \mathcal{D}_N^s} \,\rd Z_N \,\rd t = 0 . 
\end{split}
\end{equation}
Then the equation satisfied by $\tilde{P}^{(s)}(t, Z_s)$ follows from integration by parts. More precisely, we denote $X_s := (x_1,...,x_s) \in \RR^{3s}$ and $W_s := (w_1,...,w_s) \in \RR^{3s}$.

(I) For the first term on the time derivative of $P^{(N)}$, we have

\begin{equation*}
\begin{split}
    &\int_{\RR_{+} \times \RR^{6N}} \frac{\partial P^{(N)}(t,Z_N)}{\partial t} \phi(t,Z_s) \mathrm{1}_{X_N \in \mathcal{D}_N^s} \,\rd Z_N \,\rd t \\
    = &- \int_{\RR^{6N}} P^{(N)}(0,Z_N) \phi(0,Z_s) \mathrm{1}_{X_N \in \mathcal{D}_N^s} \,\rd Z_N - \int_{\RR_{+} \times \RR^{6N}} P^{(N)}(t,Z_N)\frac{\partial \phi(t,Z_s)}{\partial t}\mathrm{1}_{X_N \in \mathcal{D}_N^s} \,\rd Z_N \,\rd t\\
    =& - \int_{\RR^{6s}} \tilde{P}^{(s)}(0,Z_s) \phi(0,Z_s)  \,\rd Z_s - \int_{\RR_{+} \times \RR^{6s}} \tilde{P}^{(s)}(t,Z_s)\frac{\partial \phi(t,Z_s)}{\partial t} \,\rd Z_s \,\rd t,
\end{split}
\end{equation*}
where we notice the definition of $\tilde{P}^{(s)}$ in the second equality.\\

(II) For the second term on the spatial derivative of $P^{(N)}$, we define, for any coupled index $(i,j) \in [1,N] \times [1,N] $,
\begin{equation*}
    \Sigma_{N}^{s}(i,j):= \Big\{ X_N \in \RR^{3N}, \ |x_i - x_j| = \ve \ \Big| \ \forall (k,l) \in [1,s]\times [s+1,N] /\{i,j\}, \ |x_k-x_l|>\ve \Big\}\,,
\end{equation*}
which is a smooth sub-manifold of $\{ X_N \in \RR^{3N},\ |x_i - x_j| = \ve \}$. If denoting $\rd \sigma_N^{i,j}$ as its surface measure and $n^{i,j}$ the outward normal vector to $\Sigma_N^s(i,j)$, we obtain via integration by parts,
\begin{equation*}
    \begin{split}
        &\sum_{i=1}^{N} \int_{\RR_{+} \times \RR^{6N}} w_i\cdot \nabla_{x_i}   P^{(N)}(t,Z_N)  \phi(t,Z_s) \mathrm{1}_{X_N \in \mathcal{D}_N^s} \,\rd Z_N \,\rd t \\
        =& - \sum_{i=1}^s\int_{\RR_{+} \times \RR^{6N}} w_i P^{(N)}(t,Z_N) \cdot \nabla_{x_i}\phi(t,Z_s) \mathrm{1}_{X_N \in \mathcal{D}_N^s} \,\rd Z_N \,\rd t\\
        &+ \sum_{\substack{i,j = 1\\ i \neq  j}}^N\int_{\RR_{+} \times \RR^{3N} \times \Sigma_N^s(i,j)}n^{i,j}\cdot W_N P^{(N)}(t,Z_N)\phi(t,Z_s)\,\rd \sigma_N^{i,j} \,\rd W_N \,\rd t \\
        =& - \sum_{i=1}^s\int_{\RR_{+} \times \RR^{6N}} w_i P^{(N)}(t,Z_N)  \cdot \nabla_{x_i}\phi(t,Z_s) \mathrm{1}_{X_N \in \mathcal{D}_N^s} \,\rd Z_N \,\rd t\\
        & +(N-s) \sum_{i=1}^s\int_{\RR_{+} \times \RR^{3N} \times \Sigma_N^s(i,j)} \frac{\nu^{i,s+1}}{\sqrt{2}}\cdot (w_{s+1} - w_i) P^{(N)}(t,Z_N)\phi(t,Z_s)\,\rd \sigma_N^{i,j} \,\rd W_N \,\rd t,
    \end{split}
\end{equation*}
where $\nu^{i,j} = \frac{x_i-x_j}{|x_i-x_j|}$ and we have used the fact that  $P^{(N)}$ satisfies the permutation invariance, i.e., $P^{(N)}(t,Z_{\sigma(N)}) = P^{(N)}(t, Z_N)$.

Similarly, 
\begin{equation*}
    \begin{split}
        &\sum_{i=1}^{N} \int_{\RR_{+} \times \RR^{6N}} [A(I+tA)^{-1} x_i] \cdot \nabla_{x_i} P^{(N)}(t,Z_N)  \phi(t,Z_s) \mathrm{1}_{X_N \in \mathcal{D}_N^s} \,\rd Z_N \,\rd t \\
        =& - \sum_{i=1}^s\int_{\RR_{+} \times \RR^{6s}} \tilde{P}^{(s)}(t,Z_s) [A(I+tA)^{-1} x_i] \cdot \nabla_{x_i}\phi(t,Z_s) \,\rd Z_s \,\rd t \\
         &-\sum_{i=1}^s\int_{\RR_{+} \times \RR^{6s}} \text{Tr}[A(I+tA)^{-1}]\tilde{P}^{(s)}(t,Z_s)  \phi(t,Z_s) \,\rd Z_s \,\rd t \\
        &+ (N-s) \sum_{i=1}^s\int_{\RR_{+} \times \RR^{3N} \times \Sigma_N^s(i,j)} \frac{\nu^{i,s+1}}{\sqrt{2}}\cdot [A(I+tA)^{-1}(x_{s+1} - x_i)] P^{(N)}(t,Z_N)\phi(t,Z_s)\,\rd \sigma_N^{i,j} \,\rd W_N \,\rd t.
    \end{split}
\end{equation*}

(III) For the third term including the potential, we split the sum into two parts:
\begin{equation*}
    \begin{split}
         &\frac{1}{\ve} \sum_{i=1}^{N} \sum_{\substack{j=1\\i\neq j}}^{N}\int_{\RR_{+} \times \RR^{6N}}   \nabla_{x_i}U\left(\frac{|x_{i} - x_{j}|}{\varepsilon}\right)\cdot \nabla_{w_i} P^{(N)}  (t,Z_N) 
		 \phi(t,Z_s) \mathrm{1}_{X_N \in \mathcal{D}_N^s} \,\rd Z_N \,\rd t\\
        =& \frac{1}{\ve}    \sum_{\substack{i,j=1\\i\neq j}}^{s}\int_{\RR_{+} \times \RR^{6N}}   \nabla_{x_i}U\left(\frac{|x_{i} - x_{j}|}{\varepsilon}\right)\cdot \nabla_{w_i} P^{(N)}  (t,Z_N) 
		 \phi(t,Z_s) \mathrm{1}_{X_N \in \mathcal{D}_N^s} \,\rd Z_N \,\rd t\\
        & + \frac{1}{\ve}    \sum_{\substack{i,j=s+1\\i\neq j}}^{N}\int_{\RR_{+} \times \RR^{6N}}   \nabla_{x_i}U\left(\frac{|x_{i} - x_{j}|}{\varepsilon}\right)\cdot \nabla_{w_i} P^{(s)}  (t,Z_s) 
		 \phi(t,Z_s) \mathrm{1}_{X_N \in \mathcal{D}_N^s} \,\rd Z_N \,\rd t,
    \end{split}
\end{equation*}
where the second term on the right-hand side vanishes due to the appearance of $\nabla_{w_i}\phi(t,Z_s)$ for $i=s+1,...,N$ after integration by parts. Therefore, it becomes
\begin{equation*}
    \begin{split}
         &\frac{1}{\ve} \sum_{i=1}^{N}    \sum_{\substack{j=1\\i\neq j}}^{N}\int_{\RR_{+} \times \RR^{6N}}   \nabla_{x_i}U\left(\frac{|x_{i} - x_{j}|}{\varepsilon}\right)\cdot \nabla_{w_i} P^{(N)}  (t,Z_N) 
		 \phi(t,Z_s) \mathrm{1}_{X_N \in \mathcal{D}_N^s} \,\rd Z_N \,\rd t\\
        =& -\frac{1}{\ve}  \sum_{\substack{i,j=1\\i\neq j}}^{s}\int_{\RR_{+} \times \RR^{6N}}   \nabla_{x_i}U\left(\frac{|x_{i} - x_{j}|}{\varepsilon}\right) \cdot \nabla_{w_i} 
		 \phi(t,Z_s) \tilde{P}^{(s)}(t,Z_s) \,\rd Z_s \,\rd t.
    \end{split}
\end{equation*}

(IV) For the fourth term on the derivative of $w$:
\begin{equation*}
\begin{split}
    &\sum_{i=1}^{N} \int_{\RR_{+} \times \RR^{6N}} [A(I+tA)^{-1} w_i] \cdot \nabla_{w_i} P^{(N)}(t,Z_N) \phi(t,Z_s) \mathrm{1}_{X_N \in \mathcal{D}_N^s} \,\rd Z_N \,\rd t  \\
    =&-\sum_{i=1}^s\int_{\RR_{+} \times \RR^{6s}} \tilde{P}^{(s)}(t,Z_s) [A(I+tA)^{-1} w_i] \cdot \nabla_{w_i}\phi(t,Z_s) \,\rd Z_s \,\rd t \\
    &\qquad\qquad\qquad\qquad\qquad\qquad-\sum_{i=1}^s\int_{\RR_{+} \times \RR^{6s}} \text{Tr}[A(I+tA)^{-1}]\tilde{P}^{(s)}(t,Z_s)  \phi(t,Z_s) \,\rd Z_s \,\rd t,
\end{split}
\end{equation*}
where the integration by parts is applied in the equality above.

Combining the terms (I)-(IV) altogether, Eq.~\eqref{PN1} becomes
\begin{equation}\label{PN2}
\begin{split}
        &\int_{\RR_{+} \times \RR^{6s}} \tilde{P}^{(s)}(t,Z_s) \Big[ \partial_t \phi + \sum_{i=1}^s \left( w_i \nabla_{x_i}\phi + [A(I+tA)^{-1}x_i]\cdot \nabla_{x_i}\phi - [A(I+tA)^{-1}w_i] \cdot \nabla_{w_i}\phi \right) \\
        &\qquad \qquad - \frac{1}{\ve}\sum_{\substack{j=1\\i\neq j}}^{s} \nabla_{x_i}U\left(\frac{|x_{i} - x_{j}|}{\varepsilon}\right)\cdot  \nabla_{w_i}\phi  \Big](t,Z_s) \,\rd Z_s \,\rd t \\
        =& - \int_{\RR^{6s}} \tilde{P}^{(s)}(0,Z_s) \phi(0,Z_s)\,\rd Z_s \\
       & - (N-s) \sum_{i=1}^s \int_{\RR_{+} \times \RR^{3N} \times \Sigma_N^s(i,s+1)} \frac{\nu^{i,s+1}}{\sqrt{2}}\cdot (w_{s+1} - w_i) P^{(N)}(t,Z_N)\phi(t,Z_s)\,\rd \sigma_N^{i,s+1} \,\rd V_N \,\rd t\\
        &+ (N-s) \sum_{i=1}^s \int_{\RR_{+} \times \RR^{3N} \times \Sigma_N^s(i,s+1)} \frac{\nu^{i,s+1}}{\sqrt{2}}\cdot [A(I+tA)^{-1}(x_{s+1} - x_i)] P^{(N)}(t,Z_N)\phi(t,Z_s)\,\rd \sigma_N^{i,s+1} \,\rd V_N \,\rd t \\
        =& - \int_{\RR^{6s}} \tilde{P}^{(s)}(0,Z_s) \phi(0,Z_s)\,\rd Z_s - (N-s) \ve^{2} \int_{\RR_{+} \times \RR^{6s} } \mathcal{Q}_{s,s+1}(P^{(s+1)})(t,Z_s) \phi(t,Z_s) \,\rd Z_s \,\rd t\\
        & + (N-s) \ve^{3} \int_{\RR_{+} \times \RR^{6s} } \mathcal{Q}'_{s,s+1} (P^{(s+1)})(t,Z_s) \phi(t,Z_s) \,\rd Z_s \,\rd t, \\
\end{split}
\end{equation}
where the derivation of the term $\mathcal{Q}_{s,s+1}(P^{(s+1)})$ in the last equality directly follows \cite[Paragraph 9.3-9.4]{GST2013}, which neglects higher-order interactions except for those between binary particles:
\begin{multline*}
    \mathcal{Q}_{s,s+1}(\tilde{P}^{(s+1)})(t,Z_s):= \sum_{i=1}^s \int_{\RR^3} \int_{\mathbb{S}^{2}} b(w_i-w_{s+1},\omega_{i,s+1}) \times  \\
    \Big[ P^{(s+1)}(t,x_1,w_1,...,x_i,w'_i,...,x_s,w_s,w'_{s+1}) - P^{(s+1)}(t,Z_s,x_i,w_{s+1}) \Big] \,\rd \omega_{i,s+1} \,\rd w_{s+1},
\end{multline*}
and similarly,
\begin{multline*}
    \mathcal{Q}'_{s,s+1}(\tilde{P}^{(s+1)})(t,Z_s):= \sum_{i=1}^s \int_{\RR^3} \int_{\mathbb{S}^{2}} b\left(A(I+tA)^{-1}\frac{x_i-x_{s+1}}{|x_i-x_{s+1}|}, \omega_{i,s+1}\right) \times  \\
    \Big[ P^{(s+1)}(t,x_1,w_1,...,x_i,w'_i,...,x_s,w_s,w'_{s+1}) - P^{(s+1)}(t,Z_s,x_i,w_{s+1}) \Big] \,\rd \omega_{i,s+1} \,\rd w_{s+1}.
\end{multline*}
Here $(w'_i, w'_{s+1})$ is obtained from $(w_i, w_{s+1})$ by applying the inverse scattering operator $\sigma_{\ve}$ defined in \cite[Definition 8.2.1]{GST2013},
i.e.,
\begin{equation*}
    \sigma_{\ve}: (x_i,w_i,x_{s+1},w_{s+1}) \in \mathcal{S}_{\ve}^{-} \mapsto (x'_i,w'_i,x'_{s+1},w'_{s+1}) \in \mathcal{S}_{\ve}^{+},
\end{equation*}
where 
\begin{equation*}
    \begin{split}
        \mathcal{S}_{\ve}^{\pm} :=& \big\{ (x_i,w_i,x_{s+1},w_{s+1}) \in \mathbb{R}^{4\times 3} \, \big| \, |x_i-x_{s+1}|=\ve, \, \pm(x_i - x_{s+1})\cdot(w_i-w_{s+1}) >0  \big\},\\[2pt]
        x'_i :=& -x_i + \omega_{i,s+1}\cdot(x_i-x_{s+1})\omega_{i,s+1} + \frac{\ve \tau_*}{2}(w_i+w_{s+1}), \\[2pt]
        x'_{s+1} :=& -x_{s+1} - \omega_{i,s+1}\cdot(x_i-x_{s+1})\omega_{i,s+1} + \frac{\ve \tau_*}{2}(w_i+w_{s+1}),\\[2pt]
        w'_i :=& w_i - \omega_{i,s+1}\cdot(w_i - w_{s+1})\omega_{i,s+1}, \quad w'_{s+1} := w_{s+1} + \omega_{i,s+1}\cdot(w_i - w_{s+1})\omega_{i,s+1},
    \end{split}
\end{equation*}
with $\tau_*$ the microscopic interaction time \cite[(8.1.10)]{GST2013}.
And collision kernel $b$ has the same definition as in \cite[Definition 8.3.3 and Eq.~(8.3.5)]{GST2013}, or more specifically,
\begin{equation*}
    \nu^{i,s+1} \cdot (w_i-w_{s+1}) \,\rd \sigma_N^{i,s+1} = \frac{1}{\ve} (x_i - x_{s+1}) \cdot (w_i-w_{s+1}) \,\rd \sigma_N^{i,s+1} = \ve^{2} b(w_i - w_{s+1}, \omega_{i,s+1}) \,\rd \omega_{i,s+1}.
\end{equation*}



Taking the Boltzmann-Grad limit $N\ve^{2} \rightarrow O(1)$ (or in the case of general dimension $d$, $N\ve^{d-1} \rightarrow O(1)$ ) as $N \rightarrow \infty$ and $\ve \rightarrow 0$, the integral term involving $\mathcal{Q}'_{s,s+1}$ in Eq.~\eqref{PN2} vanishes, since $N\ve^{3} \rightarrow 0$ in such a limit. Therefore, formally $\tilde{P}^{(s)}(t,Z_s)$ satisfies the following equation in the weak sense: 
\begin{equation}\label{PsB}
\begin{split}
     &\frac{\partial \tilde{P}^{(s)}}{\partial t}  + \sum_{i=1}^s \Big( w_i\cdot \nabla_{x_i} \tilde{P}^{(s)} + [A(I+tA)^{-1}x_i] \cdot \nabla_{x_i}\tilde{P}^{(s)} - [A(I+tA)^{-1}w_i] \cdot \nabla_{w_i}\tilde{P}^{(s)}  \Big)\\
     & \qquad \qquad - \frac{1}{\ve} \sum_{\substack{i,j=1\\i \neq j}}^s \nabla_{x_i}U \left(\frac{|x_{i} - x_{j}|}{\varepsilon}\right) \cdot \nabla_{w_i}\tilde{P}^{(s)} = \mathcal{Q}_{s,s+1}(\tilde{P}^{(s+1)}).
\end{split}
\end{equation}

In particular, when $s=1$ in Eq.~\eqref{PsB}, we have the corresponding equation for one-particle distribution function $\tilde{P}^{(1)} (t, z_1)$:
\begin{multline}\label{p1t}
    \frac{\partial \tilde{P}^{(1)}(t,z_1)}{\partial t} + w_1\cdot \nabla_{x_1}\tilde{P}^{(1)}(t,z_1) + [A(I+tA)^{-1} x_1] \cdot \nabla_{x_1} \tilde{P}^{(1)}(t,z_1) 
    \\ -  [A(I+tA)^{-1} w_1] \cdot \nabla_{w_1} \tilde{P}^{(1)}(t,z_1) 
    = \mathcal{Q}_{1,2}(\tilde{P}^{(2)})(t,z_1),
\end{multline}
where the collision operator $\mathcal{Q}_{1,2}(P^{(2)})$ is
\begin{equation*}
    \mathcal{Q}_{1,2}(\tilde{P}^{(2)})(t,z_1):= \int_{\RR^3} \int_{\mathbb{S}^{2}} b(w_1-w_2,\omega_{1,2})  \Big[ \tilde{P}^{(2)}(t,x_1,w'_1,x_2,w'_2) - \tilde{P}^{(2)}(t,x_1,w_1,x_2,w_2) \Big] \,\rd \omega_{1,2} \,\rd w_{2} \,.
\end{equation*}
At this stage, one can find that Eq.~\eqref{p1t} above is still not solvable due to the existence of $\tilde{P}^{(2)}$ on the right-hand side. Hence, to close up the hierarchy, we introduce the following ``propagation of chaos" assumption \cite[p.~12]{Villani02},
\begin{equation*}
    \tilde{P}^{(2)}(t,x_1,w_1,x_2,w_2) =  \tilde{P}^{(1)}(t,x_1,w_1) \tilde{P}^{(1)}(t,x_2,w_2)\,,
\end{equation*}
which implies the un-correlation of velocities of two particles that are about to collide. Then, 
\begin{multline*}
    \mathcal{Q}_{1,2}(\tilde{P}^{(2)}) = \mathcal{Q}_{1,2}(\tilde{P}^{(1)}, \tilde{P}^{(1)}) \\
    = \int_{\RR^3} \int_{\mathbb{S}^{2}} b(w_1-w_2,\omega_{1,2})  \Big[ \tilde{P}^{(1)}(t,x_1,w'_1)  \tilde{P}^{(1)}(t,x_2,w'_2) - \tilde{P}^{(1)}(t,x_1,w_1)  \tilde{P}^{(1)}(t,x_2,w_2) \Big] \,\rd \omega_{1,2} \,\rd w_{2} \,.
\end{multline*}

Finally, after considering the homogeneity in $x$ and re-naming $w$ as $w_1$, $w_*$ as $w_2$, and $\omega_{1,2}$ as $\omega$, we arrive at the homo-energetic Boltzmann equation Eq.~\eqref{g}, where $g(t,w) = \tilde{P}^{(1)} (t,z_1) = \tilde{P}^{(1)} (t,x_1,w_1)$.

We mention that the above derivation is formal, and to make it rigorous, the main difficulty lies in the need to justify the propagation of chaos. Additionally, we need to 
construct an appropriate functional space and obtain some \textit{a priori} uniform estimates in the BBGKY hierarchy, especially with the additional term $L(t)w \cdot \nabla_w \tilde P^{(1)}$ in the homo-energetic equation. We refer to \cite[Part III]{GST2013} for the related work about the classical Boltzmann equation.

\subsection{Properties of the homo-energetic Boltzmann equation} \label{sec:prop}
In contrast to the classical Boltzmann equation Eq.~\eqref{BE}, the Eq.~\eqref{g} derived from OMD  has been effectively reduced in dimension due to spatial homogeneity, similar to the principles of microscopic molecular dynamics. Its solution, often called 
homo-energetic solution, can be regarded as a special type of solution to the full Boltzmann equation Eq.~\eqref{BE}.

The existence and uniqueness of the homo-energetic solution to Eq.~\eqref{g} in $L^1$ were initially established by C.~Cercignani in a specific instance of the deformation matrix $L(t)$ known as the shear affine flow, as discussed in \cite{Cercignani1989}. This analysis considered the collision operator with an angular cutoff cross-section, such as the hard-sphere model. In the case of the Maxwellian molecule, corresponding to $\gamma = 0$ in the kinetic part of the collision kernel $B$, the hyperbolic effect $L(t)w \cdot \nabla_{w}g$ and the collision effect $\mathcal{Q}(g,g)$ exhibit similar magnitudes. The well-posedness theory of the solution in the general case has been demonstrated within the class of Radon measures by James et al. \cite{JNV2019} or under the Fourier transform framework by Bobylev et al. \cite{BNV2020}.



One of the fundamental distinctions between the solution to the classical Boltzmann equation and the homo-energetic solution to Eq.~\eqref{g} lies in their behavior at large times. It is well known that the solution to the homogeneous Boltzmann equation converges to the global Maxwellian equilibrium, determined by the initial condition. However, for the homo-energetic equation Eq.~\eqref{g}, due to the presence of the deformation matrix $L(t)$ and its associated viscous heating effect, the equilibrium is no longer Maxwellian and the energy (or temperature) of the system steadily increases with time.

 In fact, the large-time behavior of the homo-energetic solution varies depending on the interplay between the hyperbolic term $L(t)w\cdot\nabla_{w}g$ and the collision term $\mathcal{Q}(g,g)$. In the case of the Maxwell molecule, where the collision kernel exhibits zero homogeneity, a distinct self-similar profile has been observed \cite{MT2019_Rescaled}. This self-similar distribution differs from the Maxwellian distribution and is characterized by a polynomial decay of velocity at the tails. This behavior has been numerically confirmed in previous studies \cite{GS2003}.

	More particularly, consider the self-similar transformation,
	\begin{equation}\label{self-similar-g}
	    g(t,w) = \e^{-3\beta t}G\left(\frac{w}{\e^{\beta t}}\right)\,,
	\end{equation}
	 Eq.~\eqref{g} can be re-written as
	\begin{equation}\label{steady1}
	    -\beta \nabla_{w}\cdot(wG) - \nabla_{w}\cdot (LwG) = \mathcal{Q}(G,G)\,,
	\end{equation}
	where $L \in M_{3\times 3}(\mathbb{R}), \beta \in \mathbb{R}$. Note that the intuition of self-similar scaling \eqref{self-similar-g} comes from the dimensionless analysis, see more details in \cite[p.~818]{JNV2019}. Then its well-posedness is established in the following theorem. 
 
	\begin{theorem}{\cite{JNV2019}}
	Let $B = b(\cos\theta)$ be the collision kernel for Maxwellian molecules and $\int_{-1}^{1}b(x)x^2(1-x^2)\,\rd x$ be strictly positive. There exists a sufficiently small $k_0>0$ such that, for any $\zeta \geq 0$ and any $L \in M_{3\times 3}(\mathbb{R})$ satisfying $\|L\| \leq k_0 b$, there exists $\beta \in \mathbb{R}$ and $G(w)$ that solve Eq.~\eqref{steady1} in the sense of measure and satisfy 
	\begin{equation*}
	    \int_{\mathbb{R}^3} G(w) \,\rd w = 1,\ \int_{\mathbb{R}^3} w_j G(w) \,\rd w = 0,\ 
	    \int_{\mathbb{R}^3} |w|^2 G(w) \,\rd w = \zeta\,.
	\end{equation*}
	\end{theorem}
	
	In addition to \cite{JNV2019}, the existence of the self-similar profile was also established by Bobylev, Nota, and Velazquez in \cite{BNV2020} using the Fourier method. Furthermore, in \cite{DL2021_Shear}, a smooth self-similar solution with $C^{\infty}$-regularity and dependence on a small shear force was demonstrated based on a perturbative approach.

	\begin{remark}
    It is important to note that the existence of self-similar solutions is not limited to the homo-energetic solution of Eq.~\eqref{g}, but also applies to the classical Boltzmann equation without deformation forces. However, in the absence of deformation forces, the existence of self-similar solutions is restricted by the energy conservation property. Specifically, self-similar solutions can only be demonstrated when they possess an infinite second-order moment \cite{BC2002selfsimilarapplication}. This condition holds for certain cases, such as the inelastic Boltzmann equation in granular materials \cite{BCT2003,BC2003}. Moreover, the dynamic stability of these infinite energy self-similar profiles has been established in \cite{MYZ2017convergence}.
	\end{remark}

	
\section{From Mesoscopic regime to macroscopic regime}
\label{sec:hydro}

In this section, our focus will be on investigating the hydrodynamic limit for the kinetic model induced by OMD, which bridges the gap between the mesoscopic regime and the macroscopic regime.

\subsection{Universal conservation laws}
   We begin by revisiting the macroscopic quantities that arise from classical fluid mechanics and re-formulate them within the context of homo-energetic flow \eqref{ansatz}. Notably, owing to the homogeneity in $x$ of $g(t,w)$, the macroscopic quantities will solely be time-dependent.

     Noting that in \eqref{ansatz} that if we begin with $\int_{\mathbb{R}^{3}} g(0,w) w \rd w=0$, then the first moment of $g$ remains zero, i.e., $\int_{\mathbb{R}^{3}} g(t,w) w \rd w=0$. This condition is maintained throughout the derivation, as explained in \cite[Remark 2.2]{JNV2020}. Therefore, in the subsequent analysis, we consistently assume that the first moment of $g$ is zero.

	\begin{itemize}
		\item Density $\rho(t,x)$:
		\begin{equation}\label{rho}
			\rho(t,x) = \int_{\mathbb{R}^{3}} f(t,x,v) \,\rd v = \int_{\mathbb{R}^{3}} g(t,w) \,\rd w  =: \rho(t)\,.
		\end{equation}
	
		\item The bulk velocity $u(t,x)$:
		\begin{equation*}
		\begin{split}
		    u(t,x) =& \frac{1}{\rho(t,x)}\int_{\mathbb{R}^{3}} f(t,x,v) v \,\rd v\\
		    =& \frac{1}{\rho(t)}\int_{\mathbb{R}^{3}} g(t,w) [w + L(t)x] \,\rd w \\
		    =& \frac{1}{\rho(t)}\int_{\mathbb{R}^{3}} g(t,w) w \,\rd w + [L(t)x]\frac{1}{\rho(t)}\int_{\mathbb{R}^{3}} g(t,w)   \,\rd w \\
		    =& L(t)x\,.
		\end{split}
		\end{equation*}
  

		\item Internal energy $ e(t,x) $ and temperature $\theta(t,x)$:
		\begin{equation*}
			\begin{split}
				\rho(t,x) e(t,x) =& \frac{1}{2}\int_{\mathbb{R}^{3}} f(t,x,v) |v-u(t,x)|^{2} \,\rd v \\
				 =& \frac{1}{2} \int_{\mathbb{R}^{3}} g(t,w) |w|^{2} \,\rd w\\
				 =: &\rho(t) e(t)
			\end{split}
		\end{equation*}
  
		
  Consider the equation of state for perfect gas in three dimensions that 
		\begin{equation*}
			e(t,x) = \frac{k_{B}\theta(t,x)}{\gamma_{a} - 1} = \frac{3}{2} \theta(t,x)\,,
		\end{equation*}
		where $ \theta $ is temperature, $ k_{B}=1 $ is Boltzmann constant, and $ \gamma_{a}=1+\frac{2}{d}=\frac{5}{3} $ is the adiabatic exponent. Then the temperature $ \theta (t,x) $ at position $ x $ and time $ t $ is
		\begin{equation}\label{theta}
			\begin{split}
				\rho(t,x)\theta(t,x) = & \frac{2}{3} \left( \frac{1}{2} \int_{\mathbb{R}^{3}} f(t,x,v) \left[ v-u(t,x) \right]^{2} \,\rd v \right) \\
				=& \frac{1}{3} \int_{\mathbb{R}^{3}} g(t,w) |w|^{2} \,\rd w\\
				=: &\rho(t) \theta(t)\,.
			\end{split}
		\end{equation}
  Note that the involvement of the equation of state is not an implicit assumption that makes our derivation only work for perfect gas; instead, it merely illustrates the relation between the internal energy $e$ and temperature $\theta$ to close the system, allowing us to write the macroscopic equation in the following Section \ref{subsec:hydro_boltzmann} as the time evolution for temperature for better comparison with previous results in \cite{PSJ_2022, PSJ_2023}.
	\end{itemize}
	
	\begin{itemize}
		\item The stress tensor $ S_{ij}(t,x) $. Denote $ c(t,x):=v-u(t,x) $ as the peculiar velocity, i.e., the deviation of the microscopic velocity of a molecule from the bulk velocity, then the stress tensor can be written as:
		\begin{equation}\label{Pij}
			\begin{split}
				S_{ij}(t,x) = &\int_{\mathbb{R}^{3}} c_{i}(t,x) c_{j}(t,x) f(t,x,v) \,\rd v \\
				= & \int_{\mathbb{R}^{3}} w_{i} w_{j} g(t,w) \,\rd w \\
				=: & S_{ij}(t)
			\end{split}
		\end{equation}
    for $i,j = 1,2,3$.

		\item The heat flux $ q_{i}(t,x) $:
		\begin{equation*}
			\begin{split}
				q_{i}(t,x) = &\int_{\mathbb{R}^{3}}  c_{j}(t,x) |c(t,x)|^{2} f(t,x,v) \,\rd v \\
				= & \int_{\mathbb{R}^{3}}  w_{j} |w|^{2} g(t,w) \,\rd w \\
				=: & q_{i}(t) 
			\end{split}
		\end{equation*}
    for $i= 1,2,3$.
	\end{itemize}
	

Equipped with the aforementioned notations, we can derive the conservation forms for homo-energetic flow by multiplying the collision invariants $1$, $w_j$, and $\frac{1}{2}|w|^2$ to both sides of Eq.~\eqref{g}.
	\begin{equation}\label{Macro}
		\left\{
		\begin{aligned}
			\frac{\rd }{\rd t}\rho(t) + \text{Tr}[L(t)] \rho(t)  &= 0\,, \\
			\rho(t) \left( \frac{\rd L(t)}{\rd t} + L^{2}(t) \right) &= 0\,,\\
			\rho(t)\frac{\rd e(t)}{\rd t} + \sum_{i=1}^{3}\sum_{j=1}^{3} S_{ij}(t) L_{ij}(t) &= \rho(t)\frac{\rd e(t)}{\rd t} + S(t):L(t) = 0\,,
		\end{aligned}
		\right.
	\end{equation}
where we use the standard tensor notation $S:L = \text{Tr}(S^TL) = \text{Tr}(SL^T)$. 

Eqs.~\eqref{Macro} is universally satisfied by the macroscopic quantities, irrespective of whether the scaling is mean-field or Boltzmann-Grad. The first equation represents the evolution of the density over time, while the second equation holds true for any $L(t) = A(I+tA)^{-1}$. The third equation describes the evolution of the internal energy or temperature.

	
\subsection{Hydrodynamic limit of the Boltzmann-type model}
\label{subsec:hydro_boltzmann}

Although Eqs.~\eqref{Macro} holds universally, it is not a closed system due to the absence of an explicit relationship between $e$ and $P$. In this subsection, we will address this issue by studying the hydrodynamic limit of the homo-energetic Boltzmann equation Eq.~\eqref{g}. Through asymptotic analysis, we can derive a constitutive relation that allows us to close the system.
	
Consider the homo-energetic Boltzmann equation in a dimensionless manner:
	\begin{equation}\label{gdimensionless}
		\text{St}\partial_{t} g(t,w) - \left[ L(t)w \right]\cdot \nabla_{w} g(t,w) = \frac{1}{\text{Kn}}\mathcal{Q}(g,g)(t,w),
	\end{equation}
 where $\text{Kn}$ is the Knudsen number, defined as the ratio of the mean free path to the macroscopic length scale, and $\text{St}$ is the Strouhal number, defined as the ratio between macroscopic velocity and thermal speed.
Throughout this section, we assume that $\mathcal{Q}(g,g)$ is equipped with the hard potential collision kernel $B$ under the cutoff assumption, and
	\begin{equation*}
		\text{Kn} =  \epsilon \ll 1 \quad \text{and} \quad \text{St} = 1\,,
	\end{equation*}
Then, Eq.~\eqref{gdimensionless} becomes
	\begin{equation}\label{gkn}
		\partial_{t} g(t,w) - \left[ L(t)w \right]\cdot \nabla_{w} g(t,w) = \frac{1}{\epsilon}\mathcal{Q}(g,g)(t,w).
	\end{equation}
	
\subsubsection{The Compressible Euler limit}
\label{subsec:EL}
	
We first derive the compressible Euler limit through the Hilbert expansion. Specifically, we seek the solution of Eq.~\eqref{gkn} in the form of a formal power series in $\epsilon$:
	\begin{equation}\label{hilbert}
		g_{\epsilon}(t,w) = \sum_{n\geq 0} \epsilon^{n} g_{n}(t,w) = g_{0}(t,w) + \epsilon g_{1}(t,w) + \cdots\,.
	\end{equation}
Then at $ O(\epsilon^{-1}) $, we have 
	\begin{equation*}
		\mathcal{Q}(g_{0},g_{0})(t,w ) = 0\,,
	\end{equation*}
which, by also considering the homogeneity of $\rho$, $\theta$ in $x$ from previous discussion, implies that $g_0(t,w)$ is in the form of Maxwellian distribution, i.e., 
	\begin{equation}\label{maxwellian}
		g_{0}(t,w) = \mathcal{M}_{[\rho(t), \theta(t)]} := \frac{\rho(t)}{[2\pi \theta(t)]^{\frac{3}{2}}} \e^{-\frac{|w|^{2}}{2\theta(t)}}, \quad \rho(t) > 0, ~~ \theta(t) > 0\,.
	\end{equation}
	
At $ O(\epsilon^{0}) $, we have the following equation
	\begin{equation} \label{0707}
		\Big( \partial_{t} - [L(t)w]\cdot \nabla_{w} \Big) g_{0}(t,w) = \mathcal{Q}(g_{0},g_{1})(t,w) + \mathcal{Q}(g_{1},g_{0})(t,w)\,.
	\end{equation}
 Define the linearized Boltzmann collision operator 
   \begin{equation}	 \label{0611}	\mathcal{L}_{\mathcal{M}_{[\rho,\theta]}} g := -2\mathcal{M}^{-1}_{_{[\rho,\theta]}}\mathcal{Q}(\mathcal{M}^{-1}_{_{[\rho,\theta]}},\mathcal{M}^{-1}_{_{[\rho,\theta]}}g)\,\,.
	\end{equation}
According to \cite[Theorem 3.11]{Golse2005}, it is stated that $\mathcal{L}_{{\mathcal{M}{[\rho,\theta]}}}$ is an unbounded self-adjoint nonnegative Fredholm operator. Furthermore, its null space is spanned by the collision invariants ${1, w_i, |w|^2}$, where $i = 1,2,3$. Moreover, by setting $W = \frac{w}{\sqrt{\theta(t)}}$ and referring to equation \cite[(3.64)]{Golse2005}, we can conclude that $A(W) \in \left(\text{Ker}\ \mathcal{L}_{g_0}\right)^{\perp}$, where
\begin{equation}\label{1V}
A(W) := W \otimes W - \frac{1}{3}|W|^{2}I = \frac{1}{\theta(t)} w\otimes w - \frac{1}{3}\frac{|w|^{2}}{\theta(t)}I\,.
\end{equation}

Then, Eq.~\eqref{0707} can be rewritten as 
	\begin{equation*}
		\mathcal{L}_{g_{0}} \left(\frac{g_{1}}{g_{0}}\right) = - \Big( \partial_{t} - [L(t)w]\cdot \nabla_{w} \Big) \ln g_{0}(t,w)\,.
	\end{equation*}
Upon a direct calculation, the right-hand side of the above equation can be expressed as follows:
	\begin{equation}\label{partialln}
		\begin{split}
			\Big( \partial_{t} - [L(t)w]\cdot \nabla_{w} \Big) \ln g_{0}(t,w) = &\frac{1}{\rho(t)} \Big( \partial_{t} - [L(t)w]\cdot \nabla_{w} \Big) \rho(t)  - \frac{3}{2\theta(t)} \Big( \partial_{t} - [L(t)w]\cdot \nabla_{w} \Big) \theta(t)\\
			& + \Big( \partial_{t} - [L(t)w]\cdot \nabla_{w} \Big) \left(-\frac{|w|^{2}}{2\theta(t)}\right)\\
			=& \frac{1}{\rho(t)} \partial_{t} \rho(t) - \frac{3}{2\theta(t)} \partial_{t} \theta(t) + [L(t)w]\cdot\frac{w}{\theta(t)} + \frac{|w|^{2}}{2\theta^{2}}\partial_{t} \theta(t).
		\end{split}
	\end{equation}
We can rearrange the right-hand side of Eq.~\eqref{partialln} and express it as a linear combination of ${1, w_i, |w|^2}$, where $i = 1,2,3$, and $A(W)$ in the following form:
	\begin{equation}\label{Loperator}
		\begin{split}
			-\mathcal{L}_{g_{0}}\left(\frac{g_{0}}{g_{1}}\right) = &\Big( \partial_{t} - [L(t)w]\cdot \nabla_{w} \Big) \ln g_{0}(t,w) \\
			= &\frac{1}{\rho(t)} \Big( \partial_{t}\rho(t) + \text{Tr}[L(t)] \rho(t)  \Big) + \frac{1}{2} \left( \frac{|w|^{2}}{\theta(t)}-3 \right)\frac{1}{\theta(t)} \Big( \partial_{t}\theta(t) + \frac{2}{3}\text{Tr}[L(t)] \theta(t) \Big)\\
			& +A(W):D \,,
		\end{split}
	\end{equation}
 where $ D $ is denoted as
	\begin{equation}
		D:= \frac{1}{2}\Big( L(t) + [L(t)]^{\top} - \frac{2}{3} \text{Tr}[L(t)] I \Big).
	\end{equation}
Clearly,  the last term on the right-hand side of Eq.~\eqref{Loperator} belongs to $ \left(\text{Ker}\mathcal{L}_{g_{0}}\right)^{\perp} $, while the first two terms are in $ \text{Ker}\mathcal{L}_{g_{0}} $.

Therefore, the solvability condition, as stated in \cite[(3.63)]{Golse2005}  for the Fredholm integral problem \eqref{Loperator} requires that the right-hand side of \eqref{Loperator} is perpendicular to $\text{Ker}\ \mathcal{L}_{g_0}$. This condition further implies that the coefficients of the functions $1$ and $\frac{1}{2}\left(|W|^{2} - 3\right)$ must vanish, i.e., 
	\begin{equation*}
		\left\{
		\begin{aligned}
			\partial_{t}\rho(t) + \text{Tr}[L(t)] \rho(t)  &= 0, \\
			\partial_{t}\theta(t) + \frac{2}{3}\text{Tr}[L(t)] \theta(t) & = 0\,.
		\end{aligned}
		\right.
	\end{equation*}
 Here the first equation reduces to the same equation in Eqs.~\eqref{Macro}, while the second one corresponds to the third equation in Eqs.~\eqref{Macro} with the relationship between the pressure law, internal energy and temperature given by:
    \begin{equation}\label{ideal}
    S = \rho(t)\theta(t)I \quad \text{and} \quad e(t) = \frac{3}{2}\theta(t)\,.
    \end{equation}
This system is recognized as the compressible Euler system Eqs.~\eqref{euler} in the case of the perfect monatomic gas.

\begin{remark}
We pointed out that the equilibrium \eqref{maxwellian} shares some similarities with the long-time behavior of \eqref{g}. As proposed in \cite[Section 6.1]{JNV2019}, in the collision-dominated scenario, $g(t,w)$ is expected to approach $\frac{1}{(2\pi)^{3/2}}\beta(t)^{3/2} e^{-\beta(t)|w|^2}$—a form also exhibiting Gaussian characteristics.  However, one should not confuse this conjecture with the equilibrium \eqref{maxwellian}  we use here to derive the hydrodynamic limit as they essentially represent different asymptotics.
\end{remark}


 
\subsubsection{The Compressible Navier-Stokes Limit} 
\label{subsec:NS}

We further derive the compressible Navier-Stokes equation by investigating the next order term in the asymptotic expansion. Here we follow \cite{Golse2005} and use a slightly different expansion for the solution:
	\begin{equation}\label{Chapman}
		g_{\epsilon}(t,w) = \sum_{n\geq 0} \epsilon^{n} g_{n}[\vec{P}(t)](w) = g_{0}[\vec{P}(t)](w) + \epsilon g_{1}[\vec{P}(t)](w) + \cdots\,.
	\end{equation}
 Compared to the Hilbert expansion \eqref{hilbert}, we require that $g_0$ has the same first five moments as $g_\epsilon$ by construction. That is,
	\begin{equation*}
		\int_{\mathbb{R}^{3}} g_{0} [\vec{P}(t)](w) \left(                
		\begin{array}{c}   
			1 \\  
			\frac{|w|^{2}}{2}   
		\end{array}
		\right) \,\rd w = \vec{P}(t)\,,
	\end{equation*}
 where $\vec{P}$ is a vector of conserved quantities. As a result, 
	\begin{equation}\label{gnnP}
		\int_{\mathbb{R}^{3}} g_{n} [\vec{P}(t)](w) \left(                
		\begin{array}{c}   
			1 \\  
			\frac{|w|^{2}}{2}   
		\end{array}
		\right) \,\rd w = \vec{0}\,, \quad \text{for all}\quad  n \geq 1\,.
	\end{equation}
 This expansion is termed as Chapman-Enskog expansion. 

 By taking the moments of Eq.~\eqref{gkn}, the conserved quantities satisfy a system of conservation laws:
	\begin{equation}\label{Phin}
		\partial_{t}\vec{P}(t) = \sum_{n\geq 0} \epsilon^{n}\Phi_{n}[\vec{P}](t) = \Phi_{0}(t) + \epsilon \Phi_{1}[\vec{P}](t) + \cdots\,,
	\end{equation}
where the flux term $ \Phi_{n}[\vec{P}](t) $ is denoted from the conservation law associated with Eq.~\eqref{gkn}
	\begin{equation*}
		\Phi_{n}[\vec{P}](t) = \int_{\mathbb{R}^{3}} \left(               
		\begin{array}{c}   
			1 \\  
			\frac{|w|^{2}}{2}   
		\end{array}
		\right) [L(t)w]\cdot\nabla_{w}g_{n}[\vec{P}(t)](w) \,\rd w,
	\end{equation*}
for $ n\geq 0 $.
	
As with the derivation in the previous section, at the leading order $ O(\epsilon^{0}) $, we obtain that 
	\begin{equation*}
		0 = \mathcal{Q}\left(g_{0}[\vec{P}(t)], g_{0}[\vec{P}(t)]\right)(w),
	\end{equation*}
which implies that $ g_{0}[\vec{P}(t)](w) $ is in the form of Maxwellian distribution as in Eq.~\eqref{maxwellian}. 

 At the next order $ O(\epsilon^{1}) $, we have that 
	\begin{equation}\label{g1}
		\Big( \partial_{t} - [L(t)w]\cdot \nabla_{w} \Big) g_{0}[\vec{P}(t)](w) = \mathcal{Q}\left(g_{0}[\vec{P}(t)],g_{1}[\vec{P}(t)]\right)(w) + \mathcal{Q}\left(g_{1}[\vec{P}(t)],g_{0}[\vec{P}(t)]\right)(w).
	\end{equation}
Using the form of $ g_{0}[\vec{P}(t)](w) $ in \eqref{maxwellian} and the fact that $\vec P$ solves \eqref{Phin}, the left hand side of \eqref{g1} becomes
	\begin{equation}\label{Pg0}
		\begin{split}
			\Big( \partial_{t} - [L(t)w]\cdot \nabla_{w} \Big) g_{0}[\vec{P}(t)](w)
			=\ & g_{0}[\vec{P}(t)](w) \left[ A(W):D \right] + O(\epsilon),
		\end{split}
	\end{equation}
 where the $O(\epsilon)$ term comes from the high order terms in \eqref{Phin}. 
 
 Substituting \eqref{Pg0} into \eqref{g1} and omitting the higher order term, and using the definition of linearization operator Eq.~\eqref{0611},
$ g_{1}[\vec{P}(t)](w) $ is determined by
	\begin{equation}\label{Lg1}
		\left\{
		\begin{aligned}
			\mathcal{L}_{g_{0}[\vec{P}(t)]}\left(\frac{g_{0}[\vec{P}(t)]}{g_{1}[\vec{P}(t)]}\right)  = -\left[ A(W):D \right],\\
			\int_{\mathbb{R}^{3}} g_{1} [\vec{P}(t)](w) \left(            
			\begin{array}{c}   
				1 \\  
				\frac{|w|^{2}}{2}   
			\end{array}
			\right) \,\rd w = \vec{0}.
		\end{aligned}
		\right.
	\end{equation}
and therefore $g_{1}[\vec{P}(t)]$ can be solved:
	\begin{equation*}
		\begin{split}
			 g_{1}[\vec{P}(t)] = -g_{0}[\vec{P}(t)](w) \left[ a(\theta,|W|) A(W):D \right],
		\end{split}
	\end{equation*}
	where the scalar quantity $ a(\theta,|W|) $ is denoted as $\mathcal{L}_{g_{0}[\vec{P}(t)]}(a(\theta,|W|)A(W)) = A(W)$.

	Hence, the first-order correction to the fluxes in the formal conservation law is 
	\begin{equation}\label{phi1}
		\begin{split}
			\Phi_{1}[\vec{P}(t)](w) =& \int_{\mathbb{R}^{3}} [L(t)w]\cdot \nabla_{w}g_{1} [\vec{P}(t)](w) \left(                 
			\begin{array}{c}   
				1 \\ [4pt]
				\frac{|w|^{2}}{2}   
			\end{array} 
			\right) \,\rd w\\
			 =& \left(                 
			\begin{array}{c}   
				0 \\  [4pt]
				\mu(\theta) \frac{1}{2}\Big( \text{Tr}[L^{2}(t)] + L(t):L(t) - \frac{2}{3} (\text{Tr}[L(t)])^2 \Big)
			\end{array}
			\right),
		\end{split}
	\end{equation}
	where the viscosity $ \mu(\theta) $ can be computed as in \cite[5.15]{Golse2005}:
 \begin{equation}\label{mu1}
    \mu(\theta) = \frac{2}{15}\theta \int_{0}^{\infty} a(\theta,r) r^6 \frac{1}{\sqrt{2\pi}}\e^{-r^2/2} \,\rd r\,.
 \end{equation}
	Recall Eq.~\eqref{Phin} and keeps only the first two order terms, we have 
	\begin{equation}\label{phi2}
		\partial_{t}\vec{P}(t)  =  \Phi_{0}[\vec{P}](t) + \epsilon \Phi_{1}[\vec{P}](t) \ \text{mod}\ O(\epsilon^{2})\,.
	\end{equation}
   Spelling out the flux terms, we have 
	\begin{equation} \label{NS2}
		\left\{
		\begin{aligned}
			\partial_{t}\rho(t) + \text{Tr}[L(t)] \rho(t)  &= 0, \\[4pt]
			\partial_{t}\theta(t) + \frac{2}{3}\text{Tr}[L(t)] \theta(t) & = \epsilon \mu(\theta)\frac{1}{2}\Big( \text{Tr}[L^{2}(t)] + L(t):L(t) - \frac{2}{3} (\text{Tr}[L(t)])^2 \Big),
		\end{aligned}
		\right.
	\end{equation}
    which recovers the compressible Navier-Stokes system Eqs.~\eqref{NS}. 
	This also corresponds to the Eq.~(7)-(10) in \cite{PSJ_2022} and Eq.~(29)-(30) in \cite{PSJ_2023}.


	\section*{Acknowledgement}
	\label{sec:ack}
	R.~D.~James work was supported by a Vannevar Bush Faculty Fellowship (N00014-19-1-2623)  and AFOSR (FA9550-23-1-0093). L.~Wang is partially supported by NSF grant DMS-1846854. K.~Qi is supported by grants from the School of Mathematics at the University of Minnesota. 


\appendix

\section{Preliminary results for the well-posedness of the mean-field equation}
\label{app:hypo}

\subsection{Estimates of the new characteristics}

Note that Eqs.~\eqref{XWG} can be written as the characteristic equation of the new variable $V:= (X, W)$,
\begin{equation*}
    \frac{\rd }{\rd t}V = \Psi_{\xi,\mathcal{H}[g]}(t,V),
\end{equation*}
where $\Psi_{\xi,\mathcal{H}[g]}(t,V): [0,T] \times \RR^3 \times \RR^3 \rightarrow \RR^3 \times \RR^3$ is the right-hand side of Eqs.~\eqref{XWG}. Then mean-field equation Eq.~\eqref{P1UR} becomes
\begin{equation*}
    \frac{\partial g(t,x,w)}{\partial t} + \textbf{div}(\Psi_{\xi,\mathcal{H}[g]}g)(t,x,w) = 0.
\end{equation*}

To prove the well-posedness of Eq.~\eqref{P1UR}, we first study the induced characteristic trajectories.

\begin{lemma}\label{XtWt}
    For the field $\xi(t,x,w)$ satisfying the Hypothesis~\ref{Hypo_xi} and the operator $\mathcal{H}(t,x,w) $ satisfying the Hypothesis~\ref{Hypo_H}. Given $(X_0, W_0) \in \RR^3 \times \RR^3$, there exists a unique solution $(X,W)$ to Eqs.~\eqref{XWG} in $C^1([0,T],\ \RR^3 \times \RR^3)$ with $X(0)=X_0$ and $W(0)=W_0$. In addition, there exists a constant $C_{0,T}$ depending only on $|X_0|, |W_0|, T$ such that
    \begin{equation*}
        |\big(X(t), W(t)\big)| \leq |\big(X_0, W_0\big)| \e^{tC_{0,T}}, \quad \forall t \in [0,T]. 
    \end{equation*}
\end{lemma}

\begin{proof}
    Considering the field $\xi(t,x,w)$ satisfying the Hypothesis~\ref{Hypo_xi} and the operator $\mathcal{H}(t,x,w)$ satisfying the Hypothesis~\ref{Hypo_H}, where the Lipschtz continuity holds for the dynamic equation of $X(t)$ and $W(t)$, the system Eqs.~\eqref{XWG} admits a unique solution on $[0, T)$ for each initial condition $(X(0), W(0)) \in \RR^3 \times \RR^3$ by applying the standard argument of ordinary differential equations. On the other hand, the bound can be obtained from the at-most linear growth estimate \eqref{xi_growth} and \eqref{H_growth}.
\end{proof}


\begin{lemma}[Regularity of the characteristic equation]\label{RCE}
    For any $T > 0$, assume that the field $\xi(t,x,w)$ satisfies the Hypothesis~\ref{Hypo_xi} and the operator $\mathcal{H}(t,x,w)$ satisfies the Hypothesis~\ref{Hypo_H}. Then, for any closed ball $B_R \subset \RR^3 \times \RR^3$ with $R > 0$, \\
    (i) $\Psi_{\xi,\mathcal{H}}(V)$ is bounded in the compact sets: for $V=(X,W) \in B_R$ and $t \in [0,T]$,
    \begin{equation*}
        |\Psi_{\xi,\mathcal{H}}(V)| \leq C_V, \quad \forall \, V \in B_R,
    \end{equation*}
    where the constant $C_V>0$ depends on $R$, $T$, $C_{\xi}$ and $C_{\mathcal{H}}$.\\
    (ii) $\Psi_{\xi,\mathcal{H}}(V)$ is locally Lipschitz with respect to $x,w$: for all $V_1=(X_1,W_1), V_2=(X_2,W_2)$ in $B_R$ and $ t \in [0,T]$,
    \begin{equation*}
        |\Psi_{\xi,\mathcal{H}}(V_1) - \Psi_{\xi,\mathcal{H}}(V_2)| \leq L_V |V_1 - V_2|, \quad \forall \, V_1, V_2 \in B_R,
    \end{equation*}
    where the constant $L_V > 0$ depends on $R$, $T$, $L_{\xi}$ and $L_{\mathcal{H}}$.
\end{lemma}

\begin{lemma}[Dependence of characteristic equation on $\xi$ and $\mathcal{H}$]\label{DCE}
    Assume that there are two fields $\xi_1, \xi_2$ satisfying the Hypothesis~\ref{Hypo_xi} and two operators $\mathcal{H}_1, \mathcal{H}_2$ satisfying the Hypothesis~\ref{Hypo_H}.\\
    For any point $V^0 \in \RR^3 \times \RR^3$ and $R > 0$, we assume that,
    \begin{equation*}
        |\mathcal{T}^t_{\xi_1, \mathcal{H}_1}(V^0)| \leq R, \quad |\mathcal{T}^t_{\xi_2, \mathcal{H}_2}(V^0)| \leq R, \quad \forall t \in [0,T].
    \end{equation*}
    Then, for $t \in [0,T]$, it holds that 
    \begin{equation*}
        \left| \mathcal{T}_{\xi_1, \mathcal{H}_1}(V^0) - \mathcal{T}_{\xi_2, \mathcal{H}_2}(V^0) \right| \leq \frac{\e^{tL_V}-1}{L_V} \left(  |L_{\xi_1}-L_{\xi_2}|R + \sup_{\tau \in (0,T)} \|\mathcal{H}_1(\tau, \cdot,\cdot)  - \mathcal{H}_2(\tau, \cdot,\cdot) \|_{L^\infty(B_R)} \right)\,,
    \end{equation*}
    where the constant $L_V > 0$ depends on $R$, $T$, $L_{\xi_1}$ and $L_{\mathcal{H}_1}$.
\end{lemma}

\begin{proof}
    We denote $V_i(t) = \mathcal{T}^t_{\xi_i,\mathcal{H}_i}(V^0) = (X_i(t), W_i(t))$ for $i=1,2$ and $t \in [0,T]$. These functions satisfy the characteristic system Eqs.~\eqref{XWG}: for $i=1,2$, 
    \begin{equation*}
		\left\{
		\begin{aligned}
			\frac{\rd }{\rd t} V_i(t) =&\ \Psi_{\xi_i,\mathcal{H}_i}(t,V_i(t)),\\[6pt]
			V_i(0) =&\ V^0.
		\end{aligned}
		\right.
    \end{equation*}
    Then for $t \in [0,T]$, 
    \begin{equation*}
        \begin{split}
            &|V_1(t) - V_2(t)| \\
            \leq & \int_0^t \left|\Psi_{\xi_1,\mathcal{H}_1}(\tau,V_1(\tau)) - \Psi_{\xi_2,\mathcal{H}_2}(\tau,V_2(\tau))  \right| \,\rd\tau \\
            \leq & \int_0^t \left|\Psi_{\xi_1,\mathcal{H}_1}(\tau,V_1(\tau)) - \Psi_{\xi_1,\mathcal{H}_1}(\tau,V_2(\tau))  \right| \,\rd\tau + \int_0^t \left|\Psi_{\xi_1,\mathcal{H}_1}(\tau,V_2(\tau)) - \Psi_{\xi_2,\mathcal{H}_2}(\tau,V_2(\tau))  \right| \,\rd\tau\\
            \leq & L_V\int_0^t \left|V_1(\tau) - V_2(\tau) \right| \,\rd\tau + \int_0^t |L_{\xi_1}-L_{\xi_2}|R + \|\mathcal{H}_1(\tau, \cdot,\cdot)  - \mathcal{H}_2(\tau, \cdot,\cdot) \|_{L^\infty(B_R)} \,\rd\tau.
        \end{split}
    \end{equation*}
    Finally, by the Gronwall's inequality, we obtain
    \begin{equation*}
        \begin{split}
            |V_1(t) - V_2(t)| \leq \ \frac{\e^{tL_V}-1}{L_V} \left(  |L_{\xi_1}-L_{\xi_2}|R + \sup_{\tau \in (0,T)} \|\mathcal{H}_1(\tau, \cdot,\cdot)  - \mathcal{H}_2(\tau, \cdot,\cdot) \|_{L^\infty(B_R)} \right).
        \end{split}
    \end{equation*}
\end{proof}

\begin{lemma}[Regularity of characteristics with respect to time]\label{CR_time}
    For any $T > 0$, assume that the field $\xi(t,x,w)$ satisfies the Hypothesis~\ref{Hypo_xi} and the operator $\mathcal{H}(t,x,w)$ satisfies the Hypothesis~\ref{Hypo_H}. For any initial condition $V^0 = \RR^3 \times \RR^3$ and $R > 0$ such that
    \begin{equation*}
        |\mathcal{T}^t_{\xi,\mathcal{H}}(V^0)| \leq R, \quad  \forall \ t \in [0,T]\,,
    \end{equation*}
    it holds that 
    \begin{equation*}
        |\mathcal{T}^t_{\xi,\mathcal{H}}(V^0) - \mathcal{T}^s_{\xi,\mathcal{H}}(V^0)| \leq C |t-s|, \quad \forall \ s,t \in [0,T],
    \end{equation*}
    where the constant $C > 0$ depends only on $R$, $C_{\xi}$ and $C_{\mathcal{H}}$.
\end{lemma}

\begin{proof}
    This is a direct consequence of the definition of $\mathcal{T}^t_{\xi,\mathcal{H}}(V^0)$ and the point (ii) in the Hypothesis \ref{Hypo_xi} and point (ii) in the Hypothesis \ref{Hypo_H}.
\end{proof}

\begin{lemma}[Regularity of characteristics with respect to initial condition]\label{CR_initial}
    For any $T > 0$, assume that the field $\xi(t,x,w)$ satisfies the Hypothesis~\ref{Hypo_xi} and the operator $\mathcal{H}(t,x,w)$ satisfies the Hypothesis~\ref{Hypo_H}. For two initial conditions $V^0_1, V^0_2 \in \RR^3 \times \RR^3$ and $R > 0$ such that
    \begin{equation*}
        |\mathcal{T}^t_{\xi,\mathcal{H}}(V^0_1)| \leq R, \quad  |\mathcal{T}^t_{\xi,\mathcal{H}}(V^0_2)| \leq R, \quad  \forall \ t \in [0,T]\,,
    \end{equation*}
   it holds that 
    \begin{equation*}
        |\mathcal{T}^t_{\xi,\mathcal{H}}(V^0_1) - \mathcal{T}^t_{\xi,\mathcal{H}}(V^0_2)| \leq |V^0_1 - V^0_2| \e^{t L_V}, \quad \forall \ s,t \in [0,T],
    \end{equation*}
    where the constant $L_V > 0$ depends only on $R$, $C_{\xi}$ and $C_{\mathcal{H}}$.
\end{lemma}
\begin{proof}
    We denote $V_i(t) = \mathcal{T}^t_{\xi,\mathcal{H}}(V^0_i) = (X_i(t), W_i(t))$ for $i=1,2$ and $t \in [0,T]$. These functions satisfy the characteristic system Eqs.~\eqref{XWG}: for $i=1,2$, 
    \begin{equation*}
		\left\{
		\begin{aligned}
			\frac{\rd }{\rd t} V_i(t) =&\ \Psi_{\xi,\mathcal{H}}(t,V_i(t)),\\[6pt]
			V_i(0) =&\ V^0_i.
		\end{aligned}
		\right.
    \end{equation*}
    Hence, by Lemma \ref{RCE}, we have
    \begin{equation*}
    \begin{split}
        |V_1(t) - V_2(t)| \leq& \ |V^0_1 - V^0_2| + \int_0^t \big|\Psi_{\xi,\mathcal{H}}(t,V_1(\tau)) - \Psi_{\xi,\mathcal{H}}(t,V_2(\tau)) \big| \,\rd \tau \\
        \leq & \ |V^0_1 - V^0_2| + L_V  \int_0^t \big|V_1(\tau) - V_2(\tau) \big| \,\rd \tau\,.
    \end{split}
    \end{equation*}
    Then the Gronwall inequality leads to
    \begin{equation*}
        |V_1(t) - V_2(t)| \leq |V^0_1 - V^0_2| \e^{t L_V}.
    \end{equation*}
    In other words, $\mathcal{T}^t_{\xi,\mathcal{H}}$ is actually Lipschitz continuous on the ball $B_R \subset \RR^3 \times \RR^3$ with the associated Lipschitz constant ${\rm{Lip}}_R\big[\mathcal{T}^t_{\xi,\mathcal{H}}\big] \leq \e^{tL_V}$ for $t \in [0,T]$.
\end{proof}




\subsection{Some preliminary lemmas}

This subsection is dedicated to presenting some preliminary lemmas and hypotheses that will be utilized to establish the well-posedness of Eq.~\eqref{P1UR}.
The first part focuses on the transport of probability measures along the characteristic trajectory, as demonstrated in the previous subsection.

\begin{lemma}\label{Lemma3.11}\cite[Lemma 3.11]{CCR2011}
    Let $V_1, V_2: \RR^3 \rightarrow \RR^3$ be two Borel measurable functions, and let $g \in \mathcal{V}_1(\RR^3)$. Then, 
    \begin{equation*}
        W_1(V_1 \# g, V_2 \# g) \leq \| V_1 - V_2\|_{L^{\infty}({\rm{supp}}g)}.
    \end{equation*}
\end{lemma}

\begin{lemma}\label{Lemma3.13}\cite[Lemma 3.13]{CCR2011}
Take a locally Lipschitz map $\mathcal{T}: \RR^3 \rightarrow \RR^3 $ and $f,g \in \mathcal{P}_1(\RR^3)$ with compact support contained in the ball $B_R$. Then,
\begin{equation*}
    W_1(\mathcal{T}\# f, \mathcal{T}\# g) \leq L W_1(f,g).
\end{equation*}
where $L$ is the Lipschitz constant of $\mathcal{T}$ on the ball $B_R$.
\end{lemma}

\begin{lemma}[Continuity with respect to time]\label{con_t_g}
    For any $T > 0$, assume that the field $\xi(t,x,w)$ satisfies the Hypothesis~\ref{Hypo_xi} and the operator $\mathcal{H}(t,x,w)$ satisfies the Hypothesis~\ref{Hypo_H}.\\
    For any probability measure $g \in \mathcal{P}_c(\RR^3 \times \RR^3)$ with compact support in the ball $B_R$, there exists $C >0$ depending only on $R$, $C_{\xi}$ and $C_{\mathcal{H}}$ such that, for any $t,s \in [0,T]$,
    \begin{equation*}
        W_1 \Big(\mathcal{T}^t_{\xi,\mathcal{H}}\# g, \mathcal{T}^s_{\xi,\mathcal{H}}\#g\Big) \leq C |t-s|,
    \end{equation*}
    where $\mathcal{T}^t_{\xi,\mathcal{H}}$ is defined as in \eqref{DefT}.
\end{lemma}

\begin{proof}
    Thanks to the Lemma \ref{Lemma3.11}, and the Lemma \ref{CR_time} about the continuity of characteristics with respect to time, we have
    \begin{equation*}
        W_1 \Big(\mathcal{T}^t_{\xi,\mathcal{H}}\# g(\cdot,\cdot), \mathcal{T}^s_{\xi,\mathcal{H}}\#g(\cdot,\cdot) \Big) \leq \|\mathcal{T}^t_{\xi,\mathcal{H}} - \mathcal{T}^s_{\xi,\mathcal{H}} \|_{L^{\infty}(\text{supp}f)} \leq C|t-s|,
    \end{equation*}
    where the constant $C > 0$ depends on $R$, $C_{\xi}$ and $C_{\mathcal{H}}$ as in the Lemma \ref{CR_time}
\end{proof}


Additionally, it is worth mentioning that the operator $\mathcal{H}g$ is constructed as $\mathcal{H}g = Eg - \eta(t,w) = -\nabla U*\rho_g(t,X) - A(I+tA)^{-1} W$. In order for $\mathcal{H}g$ to sufficiently satisfy Hypothesis \ref{Hypo_H}, we will refer to the following hypothesis and lemmas concerning $E[g]$ and $U$.

\begin{hypothesis}{\cite[Hypothesis 3.1]{CCR2011}}\label{Hypo_E}
    (i) $E(t,x)$ is continuous on $[0,T] \times \RR^3$.\\
    (ii) For some $C_E > 0$,
    \begin{equation}\label{E_growth}
        |E(t,x)| \leq C_E (1+|x|), \quad \forall t,x \in [0,T] \times \RR^3.
    \end{equation}
    (iii) $E$ is locally Lipschitz with respect to $x$, i.e., for any compact support set $D \subset \RR^3$, there is $L_D$ such that
    \begin{equation*}
        |E(t,x) - E(t,y)| \leq L_D|x-y|, \quad \text{for}\ t \in [0,T] \ \text{and} \ x,y \in D.
    \end{equation*}
\end{hypothesis}

More particularly, since $E(t,x)$ takes the form $E(t,x) = E[g](t,x) = \nabla U * \rho_g (t,x)$, we have the following properties.

\begin{lemma}{\cite[Lemma 3.14]{CCR2011}}
    Consider a potential $U \in C^1: \RR^3 \rightarrow \RR$ such that $\nabla U$ is locally Lipschitz and there is some constant $C>0$,
    \begin{equation*}
        |\nabla U(x)| \leq C (1+|x|), \quad \forall x \in \RR^3\,.
    \end{equation*}
    Let $g \in \mathcal{P}_1(\RR^3 \times \RR^3)$ be a probability measure with support in a ball $B_R$. Then,
    \begin{equation*}
        \| E[g] \|_{L^\infty(B_R)} \leq \|\nabla U\|_{L^\infty(B_{2R})},
    \end{equation*}
    and 
    \begin{equation*}
        {\rm{Lip}}_{R}(E[g]) \leq {\rm{Lip}}_{2R}(\nabla U).
    \end{equation*}
\end{lemma}

\begin{lemma}{\cite[Lemma 3.15]{CCR2011}}
    For $g,h \in \mathcal{P}_1(\RR^3 \times \RR^3)$ and $ R>0$, it holds that
    \begin{equation*}
        \| E[g] - E[h] \|_{L^{\infty}(B_R)} \leq {\rm{Lip}}_{2R}(\nabla U) W_1(g,h).
    \end{equation*}
\end{lemma}


\section{Derivation of the mean-field limit using empirical measure}
\label{app:delta}

Here we provide an alternative derivation of the mean-field limit using empirical measures. 
Let 
	\begin{equation*}
		g^{N}(t,x,w) = \frac{1}{N} \sum_{i=1}^{N} \delta(w-w_{i}(t)) \delta(x-x_{i}(t))
	\end{equation*}
	be the empirical measure associated with $ N $ molecules, where $\delta$ is the Dirac delta function. Then for any suitable test function $ \p(x,w) $, we have that 
	\begin{equation*}
		\begin{split}
			&\frac{\rd}{\rd t} \left\langle g^{N}(t,x,w), \p(x,w) \right\rangle_{x,w}\\ 
			= & \frac{1}{N}\sum_{i=1}^{N} \frac{\rd}{\rd t} \p(x_i(t),w_{i}(t))\\
			= & \frac{1}{N}\sum_{i=1}^{N}\nabla_{x} \p(x_i(t),w_{i}(t)) \cdot \dot{x}_{i}(t) +  \frac{1}{N}\sum_{i=1}^{N}\nabla_{w} \p(x_i(t),w_{i}(t)) \cdot \dot{w}_{i}(t) \\
			= & \underbrace{\frac{1}{N}\sum_{i=1}^{N}\nabla_{x} \p(x_i(t),w_{i}(t)) \cdot \left[ w_{i}(t) + A(I+tA)^{-1}x_{i}(t) \right]}_{:=(I)}\\
           &\underbrace{-\frac{1}{N} \sum_{i=1}^{N} \nabla_{w} \p(x_i(t),w_{i}(t)) \cdot \frac{1}{N} \sum_{l=1}^{N} \nabla_{x}U(|x_{i}(t) - x_{l}(t)|)}_{:=(II)} \\
			& \underbrace{-\frac{1}{N}\sum_{i=1}^{N} \nabla_{w} \p(x_i(t),w_{i}(t)) \cdot A(I+tA)^{-1} w_{i}(t)}_{:=(III)}\,,
		\end{split}
	\end{equation*}
    where the dynamical system Eqs.~\eqref{simu_dynamic_xw_1N} about $(\dot{x}_{i}(t), \dot{w}_{i}(t))$ is substituted in the last equality above.

   For the first term $(I)$, we have 
   \begin{equation*}
      \begin{split}
          (I) =& \frac{1}{N}\sum_{i=1}^{N}\nabla_{x} \p(x_i(t),w_{i}(t)) \cdot \left[ w_{i}(t) + A(I+tA)^{-1}x_{i}(t) \right]\\
          =& \left\langle g^{N}(t,x,w),~~ \left[ w(t) + A(I+tA)^{-1}x(t) \right] \nabla_{x}\phi(x,w)\right\rangle_{x,w}. \\
      \end{split}
   \end{equation*}
   Similarly, the third term $(III)$ rewrites as
	\begin{equation*}
	    \begin{split}
	        (III) =& - \frac{1}{N}\sum_{i=1}^{N} \nabla_{w}\p(x_i(t),w_{i}(t)) \cdot A(I+tA)^{-1} w_{i}(t)\\
	        =& -\left\langle g^{N}(t,x,w), ~~ [A(I+tA)^{-1} w] \cdot \nabla_{w}\p(x,w) \right\rangle_{x,w}. 
	    \end{split}
	\end{equation*}
    The second term $(II)$ is a bit more involved:
	\begin{equation*}
		\begin{split}
			(II) =& - \frac{1}{N} \sum_{i=1}^{N} \left[\frac{1}{N} \sum_{j=1}^{N} \nabla_{x}U(|x_{i}(t) - x_{j}(t)|) \cdot \nabla_{w}\p(x_i(t), w_{i}(t))\right]\\[2pt]
			=& -\left\langle g^{N}(t,x,w),~~ \frac{1}{N} \sum_{j=1}^{N} \nabla_{x}U(|x - x_{j}(t)|) \cdot \nabla_{w} \varphi(x,w)\right\rangle_{x,w} \\[2pt]
			=& -\left\langle g^{N}(t,x,w), ~~ \left\langle \nabla_{x}U(|x - y|),~ \frac{1}{N} \sum_{l=1}^{N}\delta(y-x_{j}(t)) \right\rangle_{y} \cdot \nabla_{w} \varphi(x,w) \right\rangle_{x,w}\\[2pt]
			=& -\left\langle g^{N}(t,x,w), ~~ \left\langle \nabla_{x}U(|x - y|), \rho_{g^{N}}(t,y) \right\rangle_{y}(t,x) \cdot \nabla_{w} \varphi(x,w) \right\rangle_{x,w}\\[4pt]
			=& -\left\langle g^{N}(t,x,w), ~~ [\nabla_{x}U*\rho_{g^{N}}](t,x)  \cdot \nabla_{w}\varphi(x,w) \right\rangle_{x,w},\\
		\end{split}
	\end{equation*}
	where $\rho_{g^N}(t,y) := \int_{\RR^3} g^N(t,y,w) \,\rd w$.
 
    Combining all terms together, we obtain the following weak form of the evolution equation for $g^N$:
    \begin{multline*}
 \Big\langle \frac{\partial g^{N}(t,x,w)}{\partial t}  - [\nabla_{x}U*\rho_{g^{N}}](t,x,w) \cdot \nabla_{w}g^{N}(t,x,w) - \nabla_w \cdot [A(I+tA)^{-1} w g^{N}(t,x,w)] \\
    + \nabla_x \cdot \Big(\left[ w+ A(I+tA)^{-1}x \right]g^{N}(t,x,w)\Big) ,\ \p(x, w) \Big\rangle_{x,w}  = 0 \,.
    \end{multline*}
    
    In the strong form, it becomes
    \begin{multline*}
    \frac{\partial g^{N}(t,x,w)}{\partial t} + \left[ w+ A(I+tA)^{-1}x \right] \cdot \nabla_{x}g^{N}(t,x,w) -  [A(I+tA)^{-1} w] \cdot \nabla_w g^{N}(t,x,w)\\
     = [\nabla_{x}U*\rho_{g^{N}}](t,x,w) \cdot \nabla_{w}g^{N}(t,x,w).
    \end{multline*}

    Then, if further considering that $g^N$ is homogeneous in $x$, it reduces to 
    \begin{align*}
        \frac{\partial g^{N}(t,w)}{\partial t}  -  [A(I+tA)^{-1} w] \cdot \nabla_w g^{N}(t,w) = 0\,.
    \end{align*}
    since the non-linear term will vanish due to the symmetry of the potential $U$,
	\begin{equation*}
	    [\nabla_{x}U*\rho_{g^{N}}](t,x,w) =  \rho_{g^{N}}(t) \int_{\mathbb{R}^3} \nabla_{x}U(|x - y|) \,\rd y = 0.
	\end{equation*}

\bibliographystyle{siam}
\bibliography{Homoenergetic_bib.bib}

\end{document}